\documentclass[a4paper,11pt]{article}
\usepackage[latin1]{inputenc}
\usepackage[english]{babel}
\usepackage{amsmath}
\usepackage{amsfonts}
\usepackage{amssymb}
\usepackage{epsfig}
\usepackage{amsopn}
\usepackage{amsthm}
\usepackage{color}
\usepackage{graphicx}
\usepackage{enumerate}
\usepackage{mathrsfs}
\usepackage{cite}
\newtheorem{theorem}{Theorem}[section]

\newtheorem{lemma}[theorem]{Lemma}
\newtheorem{proposition}[theorem]{Proposition}
\newtheorem{definition}[theorem]{Definition}

\newtheorem*{theorem*}{Theorem}
\newtheorem*{lemma*}{Lemma}
\newtheorem*{remark*}{Remark}
\newtheorem*{definition*}{Definition}
\newtheorem*{proposition*}{Proposition}
\newtheorem*{corollary*}{Corollary}
\numberwithin{equation}{section}
\newcommand{\real}{\mathbb{R}}

\let\ced=\c         

\def\x{\xi}

\def\qed{\,\unskip\kern 6pt \penalty 500
\raise -2pt\hbox{\vrule \vbox to8pt{\hrule width 6pt
\vfill\hrule}\vrule}\par}
\definecolor{darkblue}{rgb}{0.05, .05, .65}
\definecolor{darkgreen}{rgb}{0.1, .65, .1}
\definecolor{darkred}{rgb}{0.8,0,0}
\newcommand{\beqn}{\begin{equation}}
\newcommand{\eeqn}{\end{equation}}
\newcommand{\bear}{\begin{eqnarray}}
\newcommand{\eear}{\end{eqnarray}}
\newcommand{\bean}{\begin{eqnarray*}}
\newcommand{\eean}{\end{eqnarray*}}
\begin{document}

\title{\huge \bf Self-similar shrinking of supports and non-extinction for a nonlinear diffusion equation with spatially inhomogeneous strong absorption}

\author{
\Large Razvan Gabriel Iagar\,\footnote{Departamento de Matem\'{a}tica
Aplicada, Ciencia e Ingenieria de los Materiales y Tecnologia
Electr\'onica, Universidad Rey Juan Carlos, M\'{o}stoles,
28933, Madrid, Spain, \textit{e-mail:} razvan.iagar@urjc.es}
\\[4pt] \Large Philippe Lauren\ced{c}ot\,\footnote{Institut de
Math\'ematiques de Toulouse, CNRS UMR~5219, Universit\'e Paul Sabatier, F--31062 Toulouse Cedex 9, France. \textit{e-mail:}
Philippe.Laurencot@math.univ-toulouse.fr}\\ [4pt] \Large Ariel S\'{a}nchez\footnote{Departamento de Matem\'{a}tica
Aplicada, Ciencia e Ingenieria de los Materiales y Tecnologia
Electr\'onica, Universidad Rey Juan Carlos, M\'{o}stoles,
28933, Madrid, Spain, \textit{e-mail:} ariel.sanchez@urjc.es} \\[4pt]}
\date{\today}
\maketitle

\begin{abstract}
We study the dynamics of the following porous medium equation with strong absorption
$$
\partial_t u=\Delta u^m-|x|^{\sigma}u^q,
$$
posed for $(t,x)\in(0,\infty)\times\real^N$, with $m>1$, $q\in(0,1)$ and $\sigma>2(1-q)/(m-1)$. Considering the Cauchy problem with non-negative initial condition $u_0\in L^{\infty}(\real^N)$, \emph{instantaneous shrinking} and \emph{localization of supports} for the solution $u(t)$ at any $t>0$ are established. With the help of this property, \emph{existence and uniqueness} of a non-negative compactly supported and radially symmetric \emph{forward self-similar solution} with algebraic decay in time are proven. Finally, it is shown that finite time extinction does not occur for a wide class of initial conditions and this unique self-similar solution is the \emph{pattern for large time behavior} of these general solutions.
\end{abstract}

\smallskip


\smallskip


\section{Introduction and main results}

The degenerate diffusion equation
\begin{equation}\label{eq1}
\partial_t u-\Delta u^m+|x|^{\sigma}u^q=0, \qquad (t,x)\in(0,\infty)\times\real^N,
\end{equation}
features, in the range of exponents
\begin{equation}\label{range.exp}
m>1, \qquad q\in(0,1), \qquad \frac{2(1-q)}{m-1}<\sigma<\infty,
\end{equation}
a competition between the nonlinear diffusion in the form of a porous medium equation and a strong absorption weighted with a spatially inhomogeneous coefficient having at least at a formal level different effects at points $x\in\real^N$ with $|x|$ small and with $|x|$ large. Understanding the implications of this competition between the two terms on the dynamics of Eq.~\eqref{eq1} is the goal of the present paper.

It has been shown since long that the effects of the competition between the nonlinear diffusion and the absorption terms on the qualitative properties of solutions to the spatially homogeneous absorption-diffusion equation
\begin{equation}\label{eq1.hom}
\partial_tu-\Delta u^m+u^q=0, \qquad m>1,
\end{equation}
depend strongly on the absorption exponent $q>0$. Indeed, if $q>m$ the influence of the absorption terms on the large time behavior of solutions to Eq. \eqref{eq1.hom} is limited and either the porous medium equation dominates as $t\to\infty$ for $q\geq m+2/N$, or there is a balance between the two terms leading to new asymptotic profiles in the form of very singular self-similar solutions, in the range $m<q<m+2/N$, as it was established in a number of previous works \cite{KP86, PT86, KU87, KV88, KPV89, Le97, Kwak98}, some of them being devoted to the construction of the self-similar profiles and others to the proof of the convergence of general solutions to Eq.~\eqref{eq1.hom} towards these profiles as $t\to\infty$.

The situation becomes more complex as the absorption exponent $q>0$ gets smaller, for the absorption starts to dominate in the dynamics of the equation. The range $1<q<m$ has been addressed also in a number of works either for Eq.~\eqref{eq1.hom} in a bounded domain \cite{BNP82} or in the whole space in relation to the existence of unbounded self-similar solutions with a specific growth at infinity \cite{MPV91} or non-uniqueness of solutions by constructing a wave coming from infinity \cite{CV96}. All these phenomena, which do not show up in the range $q>m$, are due to the strong influence of the absorption term for $q\in (1,m)$, but probably the most striking effect is seen on the large time behavior of solutions to Eq.~\eqref{eq1.hom}. This asymptotic behavior is described in dimension $N=1$ by Chaves and V\'azquez \cite{CV99}, where it is proved that for compactly supported non-negative initial conditions, the support of solutions remains localized in a large ball $B(0,R)$ and the asymptotic pattern is given by a combination between a flat solution of the form $K_*t^{-1/(q-1)}$ with $K_*>0$ explicit (which comes only from the differential equation obtained by neglecting completely the diffusion) and a boundary layer in the form of a localized self-similar solution which approximates the behavior near the moving interface. This is an example of \emph{asymptotic simplification} induced by the dominating absorption. The critical case $q=m$ has been addressed in \cite{CVW97}.

Entering the range $q\in(0,1)$ is going one step forward towards a very strong absorption effect. Two new phenomena occur with respect to the solutions to Eq. \eqref{eq1.hom} for $0<q<1$:

$\bullet$ on the one hand, \emph{finite time extinction} of (non-negative bounded) solutions occurs, that is, there exists $T\in(0,\infty)$ such that $u(t)\not\equiv0$ for $t\in(0,T)$ but $u(T)\equiv0$. This is obviously a consequence of the strong absorption, as it stems from the ordinary differential equation $\partial_t u=u^q$ obtained by neglecting the diffusion;

$\bullet$ on the other hand, \emph{instantaneous shrinking} and \emph{localization of supports} of solutions to Eq.~\eqref{eq1.hom} with bounded initial condition $u_0$ such that $u_0(x)\to0$ as $|x|\to\infty$ take place, as shown by Kalashnikov \cite{Ka74}, Evans \& Knerr \cite{EK79} and Abdullaev \cite{Abd98}. This means that for any non-negative initial condition $u_0\in L^{\infty}(\real^N)$ such that $u_0(x)\to0$ as $|x|\to\infty$ and $\tau>0$, there is $R(\tau)>0$ such that ${\rm supp}\,u(t)\subseteq B(0,R(\tau))$ for all $t\ge\tau$. This is in general a striking phenomenon due to the strength of the absorption term, which involves a very quick loss of mass of the solution.

A description of the extinction rates and behavior near the extinction time of the solutions to Eq.~\eqref{eq1.hom} seems to be only available when $m+q=2$ in \cite{GV94}, revealing a quite unusual case of asymptotic simplification, and seems to be a very complicated problem if $m+q\neq2$. It was then noticed in \cite{GSV99a, GSV99b} that the behavior near the interface depends on the sign of $m+q-2$, a combination that appears to be critical for Eq.~\eqref{eq1.hom}, although a better understanding of its dynamics is still lacking.

The case of spatially dependent strong absorption is much less studied in the literature due to its difficulty when dealing with a variable (and unbounded) coefficient. We mention here Belaud's work \cite{Belaud01} which is closely related to our study, where finite time extinction of solutions to Eq. \eqref{eq1} (but posed on a bounded domain) is proved, provided $\sigma<2(1-q)/(m-1)$, that is, exactly the complementary case to the present work. Although we consider Eq. \eqref{eq1} posed in $\real^N$, our analysis strongly suggests that the range of $\sigma\in(0,2(1-q)/(m-1))$ considered in \cite{Belaud01} is sharp with respect to the extinction phenomenon. We end this presentation by referring to a series of very recent papers by two of the authors \cite{IS20, IS22, IMS22} concerning Eq. \eqref{eq1} with strong inhomogeneous source terms instead of strong absorption, where self-similar solutions are classified by employing dynamical systems techniques and where the sign of $m+q-2$ on the one hand and the position of $\sigma$ with respect to $2(1-q)/(m-1)$ on the other hand, are highly critical.

\medskip

\noindent \textbf{Main results}. We consider the Cauchy problem for Eq.~\eqref{eq1} with non-negative and bounded initial condition
\begin{equation}\label{init.cond}
u(0)=u_0\in L_+^{\infty}(\real^N) := \big\{ z\in L^\infty(\real)\ :\ z(x)\ge 0 \;\text{ a.e. in }\; \real^N \big\}.
\end{equation}
Let us stress here that this is a wider class of initial data than the one considered in the standard $L^1$-theory for the porous medium equation, and we shall see that the strength of the absorption effect allows us to handle such data. Throughout this paper, we consider weak solutions to Eq.~\eqref{eq1}, the rigorous definition being given at the beginning of Section~\ref{sec.wp} below. Our first result states the well-posedness of the Cauchy problem~\eqref{eq1}, \eqref{init.cond} together with a property that justifies the extension of the theory to bounded but not necessarily integrable solutions. We mention here that the next result is valid for any $\sigma>0$, not only in the range of exponents~\eqref{range.exp} which introduces a restriction in $\sigma$ needed in the other main results.

\begin{theorem}[Well-posedness and instantaneous shrinking]\label{th.wp}
For any $m>1$, $q\in(0,1)$ and $\sigma>0$, there is a unique non-negative weak solution to the Cauchy problem~\eqref{eq1}, \eqref{init.cond} which satisfies
\begin{equation}
		\|u(t)\|_\infty \le \|u_0\|_\infty\,, \qquad t\ge 0. \label{wp0}
\end{equation}
In addition, it enjoys the properties of \emph{instantaneous shrinking} and \emph{localization} of the support; that is, for any $t>0$, $u(t)$ has compact support and, given $\tau>0$, there exists $R=R(\tau)>0$, depending on $u_0$ and $\tau$ but not on $t\in [\tau,\infty)$, such that
	$$
	{\rm supp}\,u(t)\subseteq B(0,R(\tau)), \qquad {\rm for \ any} \ t\geq\tau.
	$$
Also, the following \emph{comparison principle} holds true: given $u_{0,i}\in L_+^\infty(\real^N)$, $i=1,2$, such that $u_{0,1}\le u_{0,2}$ in $\real^N$, the corresponding non-negative weak solutions $u_1$ and $u_2$ to~\eqref{eq1}, \eqref{init.cond} satisfy $u_1\le u_2$ in $(0,\infty)\times\real^N$.
\end{theorem}

The main issue to be dealt with in Theorem~\ref{th.wp} is that the initial condition is only locally integrable on $\real^N$, so that the usual properties of the porous medium equation in $L^1(\real^N)$ cannot be used. Nevertheless, the well-posedness in $L_{\text{loc}}^1(\real^N)$ of the porous medium equation with spatially homogeneous nonlinear absorption is established in \cite{Gl1993, GG2002, VW1994}, while a spatially dependent nonlinear absorption vanishing in a neighborhood of $x=0$ is studied in \cite{Gl2001}. As we shall see in the proof given in Section~\ref{sec.wp}, the $L^\infty$-functional framework we are using here allows us to use simpler tools with respect to the above mentioned works.

\medskip

\noindent We also discover that, even if starting with an initial condition which is only bounded, the unique weak solution to the Cauchy problem~\eqref{eq1}, \eqref{init.cond} becomes immediately compactly supported. This phenomenon of instantaneous shrinking of supports has been noticed in a number of equations involving a strong absorption effect but only for initial conditions either converging to zero in some sense as $|x|\to\infty$ or having suitable integrability properties. Thus, in our case \emph{an improvement on the class of initial data} allowing for instantaneous shrinking of supports is achieved: indeed, mere boundedness of $u_0$ is enough for this property of localization to hold true. As we will notice from the proof, this fact is due to the influence of the weight $|x|^{\sigma}$ with $\sigma>0$: in the spatially homogeneous case $\sigma=0$ of Eq.~\eqref{eq1}, this is not true and one needs to ask that $u_0(x)\to0$ as $|x|\to\infty$, as shown in \cite{Abd98, EK79}. Since Theorem~\ref{th.wp} consists of two different results (well-posedness and shrinking and localization of supports), its proof will be split into two parts which are to be found in Section~\ref{sec.wp} and Section~\ref{sec.nsloc}, respectively.

Having settled the well-posedness of the Cauchy problem~\eqref{eq1}, \eqref{init.cond}, we move forward to studying more specialized properties of its solutions. One of the specific features of nonlinear diffusion equations is the availability of a class of special solutions, in self-similar form, which are natural candidates for the large time behavior of wider classes of (more) general solutions. Regarding Eq.~\eqref{eq1}, we look for radially symmetric non-negative self-similar solutions in forward form, with algebraic time decay but without finite time extinction, that is
\begin{equation}\label{SSS}
u(t,x)=t^{-\alpha}f(|x|t^{\beta}), \qquad (t,x)\in (0,\infty)\times\real^N.
\end{equation}
Inserting the ansatz~\eqref{SSS} into Eq.~\eqref{eq1} and letting $\xi=|x|t^{\beta}$, we get that the self-similarity exponents are given by
\begin{equation}\label{SSexp}
\alpha=\frac{\sigma+2}{\sigma(m-1)+2(q-1)}>0, \qquad \beta=\frac{m-q}{\sigma(m-1)+2(q-1)}>0
\end{equation}
as $m$, $q$ and $\sigma$ satisfy \eqref{range.exp}, while the profile $f$ solves the ordinary differential equation
\begin{equation}\label{SSODE}
(f^m)''(\xi)+\frac{N-1}{\xi}(f^m)'(\xi)+\alpha f(\xi)-\beta\xi f'(\xi)-\xi^{\sigma}f^q(\xi) =0,
\end{equation}
with initial conditions
\begin{equation}\label{init.cond.ODE}
f(0)=a, \qquad f'(0)=0,
\end{equation}
for some parameter $a>0$, assuming \textit{a priori} that $f$ does not feature a singularity at zero. With this notation, we are now in a position to state our second theorem.
\begin{theorem}[Existence and uniqueness of self-similar solutions]\label{th.uniqSS}
There exists a \emph{unique compactly supported non-negative self-similar solution} in the form \eqref{SSS} to Eq.~\eqref{eq1}
\begin{equation*}
U(t,x)=t^{-\alpha} f^{*}(|x|t^{\beta}), \qquad (t,x)\in(0,\infty)\times\real^N,
\end{equation*}
with $f^{*}$ solution to the Cauchy problem \eqref{SSODE}-\eqref{init.cond.ODE} for some $a^*=f^{*}(0)>0$.
\end{theorem}

The proof of Theorem~\ref{th.uniqSS} is based on a combination of several techniques related to the local behavior of the solutions to the Cauchy problem \eqref{SSODE}-\eqref{init.cond.ODE} near $\xi=0$ and also near the interface point $\xi=\xi_0(a)\in(0,\infty)$ (if any): it involves a shooting method for solutions to differential equations, some dynamical system techniques and the scaling and sliding technique in order to show the monotonicity of solutions with respect to the shooting parameter $a=f(0)\in(0,\infty)$ and then the uniqueness of the solution with compact support among them. All this program is developed in detail in the longest Section~\ref{sec.SSS} of the paper. Let us stress here that, together with the proof, we provide the precise local behavior of the profile $f^{*}$ when $\xi\to 0$ and near the interface point, and the vanishing of $f^{*}$ at the interface turns out to depend strongly on whether $m+q\ge 2$ or $m+q<2$.

Making use of the previous results, it is now rather straightforward to establish the local behavior of a wide class of solutions to the Cauchy problem~\eqref{eq1}, \eqref{init.cond}.

\begin{theorem}[Large time behavior of bounded solutions]\label{th.asympt}
Let $u$ be the solution to the Cauchy problem~\eqref{eq1}, \eqref{init.cond} and assume that there is $\delta>0$ and $r>0$ such that $u_0(x)\ge \delta$ for a.a. $x\in B(0,r)$. Then the behavior as $t\to\infty$ of $u(t)$ is given by the unique self-similar solution $U(t)$, in the sense that
\begin{equation}\label{asympt.beh}
\lim\limits_{t\to\infty}t^{\alpha}\|u(t)-U(t)\|_{\infty}=0.
\end{equation}
\end{theorem}

The outcome of Theorem~\ref{th.asympt} fails to be true without the positivity assumption on $u_0$ in a small neighborhood of the origin. Indeed, the availability of a stationary solution of the form $A|x|^{(\sigma+2)/(m-q)}$, together with the comparison principle, entails that, at least for some initial conditions $u_0$ as in \eqref{init.cond}, the solution to the Cauchy problem~\eqref{eq1}, \eqref{init.cond} satisfies $u(t,0)=0$ for any $t>0$. This is made precise in Proposition~\ref{prop.stat} at the end of Section~\ref{sec.asympt}.

Let us point out here an interesting outcome of our analysis: finite time extinction does not occur for solutions to the Cauchy problem~\eqref{eq1}, \eqref{init.cond} as in the statement of Theorem~\ref{th.asympt} and in the range of exponents~\eqref{range.exp}. We have thus shown evidence of a rather unexpected case of an equation enjoying \emph{instantaneous shrinking of supports but without finite time extinction}. This behavior is due to the effect of the variation with respect to the space variable of the weight $|x|^{\sigma}$: indeed, on the one hand, for $|x|$ sufficiently large the absorption is very strong, leading to a very quick loss of mass and a shrinking and localization of the support. On the other hand, if there is enough mass in a small neighborhood of the origin, at those points $|x|^{\sigma}$ is very small and the absorption is quite weak, allowing (together with the slow diffusion) thus some mass to ``survive" at any time $t>0$ and preventing finite time extinction. The proof of Theorem~\ref{th.asympt} is performed in Section~\ref{sec.asympt}.

\section{Well-posedness}\label{sec.wp}

This section is devoted to the proof of the well-posedness in $L^\infty_+(\real^N)$ of the initial value problem
\begin{equation}\label{wp1}
	\begin{split}
		\partial_t u - \Delta u^m + |x|^\sigma u^q & = 0 \;\;\text{ in }\;\; (0,\infty)\times\real^N, \\
		u(0) & = u_0 \;\;\text{ in }\;\; \real^N\,,
	\end{split}
\end{equation}
when the parameters $(m,q,\sigma)$ satisfy
\begin{equation*}
	m\ge 1\,, \quad q>0\,, \quad \sigma>0\,, 
\end{equation*}
and $u_0\in L_+^\infty(\real^N)$. Let us stress here that we do not need the further restriction on $\sigma$ given in~\eqref{range.exp}. We first state the definition of weak solutions.

\begin{definition}\label{def.wp}
A non-negative weak solution to~\eqref{wp1} is a function
\begin{subequations}\label{wp4}
\begin{equation}
	u \in L_+^\infty((0,\infty)\times\real^N) \label{wp4a}
\end{equation}
such that, for all $T>0$,
\begin{equation}
	u^m \in L^2\big((0,T),H^1_{{\rm loc}}(\real^N)\big) \label{wp4b}
\end{equation}
and
\begin{equation}
	\int_0^T \int_{\real^N} \Big[ (u_0-u) \partial_t \zeta + \nabla u^m \cdot \nabla\zeta + |x|^\sigma u^q \zeta \Big]\ dxds = 0 \label{wp4c}
\end{equation}
for all $\zeta\in C_c^1([0,T)\times\real^N)$.
\end{subequations}
\end{definition}

Once settled the notion of solution to work with, we can prove the well-posedness of the Cauchy problem~\eqref{eq1}, \eqref{init.cond}.

\begin{proof}[Proof of Theorem~\ref{th.wp}: Uniqueness and comparison principle]
	Consider $u_{0,i}\in L^\infty_+(\real^N)$, $i=1,2$, and denote the corresponding non-negative weak solution to~\eqref{wp1} by $u_i$. Setting $v_i := u_i^m$, $i=1,2$, we proceed as in the proof of \cite[Theorem, Eq.~(15)]{Ot1996} to show that, for all $T>0$ and non-negative $\vartheta\in C_c^\infty([0,T)\times\real^N)$,
	\begin{align*}
		& \int_0^T \int_{\real^N} \left\{ \Big[ (u_{0,1} - u_{0,2})_+ - (u_1-u_2)_+ \Big] \partial_t \vartheta + \mathrm{sign}_+(v_1-v_2) \nabla (v_1-v_2)\cdot \nabla\vartheta \right\}\ dxds \\
		& \hspace{5cm}  +  \int_0^T \int_{\real^N} \mathrm{sign}_+(u_1-u_2)  \big( u_1^q - u_2^q \big)\vartheta \ dxds \le 0\,.
	\end{align*}
	Owing to the monotonicity of $z\mapsto z^q$ and the non-negativity of $\vartheta$, we readily deduce from the above inequality that
	\begin{equation*}
		\int_0^T \int_{\real^N} \left\{ \Big[ (u_{0,1} - u_{0,2})_+ - (u_1-u_2)_+ \Big] \partial_t \vartheta + \nabla (v_1-v_2)_+\cdot \nabla\vartheta \right\}\ dxds \le 0\,,
	\end{equation*}
	and a further integration by parts leads us to
	\begin{equation}
		\int_0^T \int_{\real^N} \left\{ \Big[ (u_{0,1} - u_{0,2})_+ - (u_1-u_2)_+ \Big] \partial_t \vartheta - (v_1-v_2)_+ \Delta\vartheta \right\}\ dxds \le 0\,. \label{wp5}
	\end{equation}
Now, pick $T>0$, $k>N$, and a non-negative function $\xi\in C_c^\infty([0,T))$. Setting $\varrho_k(x) := (1+|x|^2)^{-k/2}$, $x\in\real^N$, we observe that
\begin{equation}
	|\nabla\varrho_k(x)| \le k \varrho_k(x)\,, \quad |\Delta\varrho_k(x)| \le k (k+2+N) \varrho_k(x)\,, \qquad x\in\real^N\,, \label{wp100}
\end{equation}
while a standard approximation argument allows us to take $\vartheta=\xi\varrho_k$ in~\eqref{wp5}, thereby obtaining
\begin{align*}
	\int_0^T  \partial_t \xi \int_{\real^N}&\Big[ (u_{0,1} - u_{0,2})_+ - (u_1-u_2)_+ \Big] \varrho_k\ dxds  \le \int_0^T \xi \int_{\real^N}  (v_1-v_2)_+ \Delta\varrho_k\ dxds \\
	& \le k(k+2+N) \int_0^T \xi \int_{\real^N}  (v_1-v_2)_+ \varrho_k\ dxds\,.
\end{align*}
We next choose appropriate approximations $\xi\in C_c^\infty([0,T))$ of $\mathbf{1}_{[0,t]}$, $t\in (0,T)$, and finally find
\begin{align*}
	\int_{\real^N} (u_1-u_2)_+(t) \varrho_k\ dx & \le \int_{\real^N} (u_{0,1} - u_{0,2})_+ \varrho_k\ dx \\
	& \qquad + k(k+2+N) \int_0^t \int_{\real^N}  (v_1-v_2)_+ \varrho_k\ dxds\,.
\end{align*}
Introducing
\begin{equation*}
	M(T) := \sup_{s\in [0,T]} \big\{ \|u_1(s)\|_\infty + \|u_2(s)\|_\infty \big\}\,,
\end{equation*}
which is finite according to Definition~\ref{def.wp}, we note that
\begin{equation*}
	(v_1-v_2)_+ = \big( u_1^m - u_2^m \big)_+ \le m \max\big\{ u_1^{m-1},u_2^{m-1}\big\} (u_1-u_2)_+ \le m M(T)^{m-1}  (u_1-u_2)_+
\end{equation*}
in $(0,T)\times \real^N$, since $m\ge 1$. Consequently,
\begin{align*}
	\int_{\real^N} (u_1-u_2)_+(t) \varrho_k\ dx & \le \int_{\real^N} (u_{0,1} - u_{0,2})_+ \varrho_k\ dx \\
	& + mk(k+2+N) M(T)^{m-1} \int_0^t \int_{\real^N}  (u_1-u_2)_+(s) \varrho_k\ dxds\,.
\end{align*}
Gronwall's lemma then gives
\begin{equation*}
	\int_{\real^N} (u_1-u_2)_+(t) \varrho_k\ dx \le e^{mk(k+2+N) M(T)^{m-1}t} \int_{\real^N} (u_{0,1} - u_{0,2})_+ \varrho_k\ dx \,, \qquad t\in [0,T]\,,
\end{equation*}
from which both uniqueness and comparison principle follow.
\end{proof}

\begin{proof}[Proof of Theorem~\ref{th.wp}: Existence] The existence proof we sketch below has some common features with previous results in literature dealing with initial conditions in $L_{\mathrm{loc}}^1(\real^N)$ (such as \cite{Gl1993, GG2002, Gl2001,VW1994}), but the $L^\infty$-functional framework allows us to use simpler tools, in particular integrable test functions instead of general compactly supported ones. Given $u_0\in L^1(\real^N)\cap L^\infty_+(\real^N)$, the existence of a non-negative finite energy solution
\begin{align*}
	& u \in L^\infty((0,\infty),L^1(\real^N))\cap L^\infty((0,\infty)\times\real^N) \cap L^{m+q}((0,\infty)\times\real^N,|x|^\sigma dxdt) \\
	& \nabla u^m \in L^2((0,\infty)\times \real^N,\real^N)
\end{align*}
to~\eqref{wp1} satisfying the properties listed in~Definition~\ref{def.wp} as well as~\eqref{wp0} is established by a standard variational approach based upon the solvability of the implicit time scheme
\begin{equation*}
	v_{n+1}^{1/m} - \tau \Delta v_{n+1} + \tau |x|^\sigma v_{n+1}^{q/m} = v_n^{1/m} \qquad {\rm in} \ \real^N\,, \qquad n\ge 0\,,
\end{equation*}
with $v_0:=u_0^m$ and $\tau>0$, which is guaranteed by the direct method of the calculus of variations applied to the energy functional
\begin{equation*}
	\int_{\real^N} \left( \frac{m}{m+1} v^{(m+1)/m} + \frac{\tau}{2} |\nabla v|^2 + \frac{m\tau}{m+q} |x|^\sigma v^{(m+q)/m}(x) - v_n^{1/m} v \right)\ dx\,.
\end{equation*}	
The uniqueness of such a solution is then a consequence of the already proven uniqueness of weak solutions to~\eqref{wp1}.

\medskip

Consider now $u_0\in L^\infty_+(\real^N)$ and define
\begin{equation*}
	u_{0,j} := u_0 \mathbf{1}_{B(0,j)}\,, \qquad j\ge 1\,.
\end{equation*}
Then
\begin{equation}
	u_{0,j} \in L^1(\real^N)\cap L_+^\infty(\real^N) \qquad {\rm with} \ u_{0,j} \le u_{0,j+1}\,, \qquad j\ge 1\,, \label{wp10}
\end{equation}
and we denote the finite energy solution to~\eqref{wp1} with initial condition $u_{0,j}$ by $u_j$. Owing to \eqref{wp0}, \eqref{wp10}, and the comparison principle,
\begin{equation}
	0 \le u_j(t,x) \le u_{j+1}(t,x) \le M := \|u_0\|_\infty\,, \qquad (t,x)\in [0,\infty)\times\real^N\,, \label{wp11}
\end{equation}
so that the function
\begin{equation}
	u(t,x) := \sup_{j\ge 1} \{u_j(t,x)\} = \lim_{j\to\infty} u_j(t,x)\,, \qquad (t,x)\in [0,\infty)\times\real^N\,, \label{wp12}
\end{equation}
is well-defined and satisfies~\eqref{wp0}.  We are left with checking that $u$ is a weak solution to~\eqref{def.wp}. As in the uniqueness proof, pick $k>N$ and recall that $\varrho_k(x) = (1+|x|^2)^{-k/2}$, $x\in\real^N$. It follows from~\eqref{wp1}, \eqref{wp100}, and~\eqref{wp11} that
\begin{align*}
	\frac{1}{m+1} \frac{d}{dt} \int_{\real^N} u_j^{m+1} \varrho_k\ dx & + \int_{\real^N} |x|^\sigma u_j^{m+q} \varrho_k\ dx \\
	& = - \int_{\real^N} \nabla \big(u_j^m \varrho_k)\cdot \nabla u_j^m\ dx \\
	& = - \int_{\real^N} |\nabla u_j^m|^2 \varrho_k\ dx - \int_{\real^N} u_j^m \nabla u_j^m \cdot \nabla\varrho_k\ dx \\
	& = - \int_{\real^N} |\nabla u_j^m|^2 \varrho_k\ dx + \frac{1}{2} \int_{\real^N} u_j^{2m} \Delta\varrho_k\ dx \\
	& \le - \int_{\real^N} |\nabla u_j^m|^2 \varrho_k\ dx + \frac{k(k+2+N)}{2} \int_{\real^N} u_j^{2m} \varrho_k\ dx \\
	& \le - \int_{\real^N} |\nabla u_j^m|^2 \varrho_k\ dx + \frac{k(k+2+N)M^{m-1}}{2} \int_{\real^N} u_j^{m+1} \varrho_k\ dx\,.
\end{align*}
Consequently, for all $T>0$, there is a positive constant $c_1(T)$ depending only on $N$, $m$, $k$, and $M$ such that
\begin{equation}
	\int_0^T \int_{\real^N} |\nabla u_j^m|^2 \varrho_k\ dxdt + \int_0^T \int_{\real^N} |x|^\sigma u_j^{m+q} \varrho_k\ dxdt \le c_1(T)\,. \label{wp13}
\end{equation}
On the one hand, it readily follows from~\eqref{wp12}, \eqref{wp13}, and a lower semicontinuity argument that $\sqrt{\varrho_k} \nabla u^m$ belongs to $L^2((0,T)\times\real^N)$ for all $T>0$. Thus, $u$ satisfies~\eqref{wp4b}. On the other hand, we proceed as in \cite{Gl1993, GG2002, VW1994} to check that $u$ satisfies the weak formulation~\eqref{wp4c} of~\eqref{wp1} and complete the proof.
\end{proof}

\section{Instantaneous shrinking and localization of supports}\label{sec.nsloc}

This section is devoted to the second statement in Theorem~\ref{th.wp}, concerning the instantaneous shrinking and localization of support for bounded, non-negative solutions to Eq.~\eqref{eq1}.

\begin{proof}[Proof of Theorem~\ref{th.wp}: instantaneous shrinking and localization of supports]
We argue as in \cite[Section~4]{EK79} and \cite[Section~2]{Ka84}. Set
$$
\mathcal{L}w:=\partial_tw-\Delta w^m+|x|^{\sigma}w^q
$$
and introduce the one-variable functions
\begin{align*}
Y_R(x) & :=[A(R)(x_1-2R)^2]^{m/(m-q)}, \qquad x = (x_1,x') \in(R,\infty)\times\real^{N-1}, \\
& \ = A(R)^{m/(m-q)} |x_1-2R|^{2m/(m-q)},
\end{align*}
and
$$
Z_R(t)=[B(R)(T-t)]^{m/(1-q)}, \qquad t\in(0,T),
$$
with $T>0$ arbitrary, $R>0$, $A(R)>0$ and $B(R)>0$ to be determined later. Introducing
$$
W_R(t,x) := \big[ Y_R(x)+Z_R(t) \big]^{1/m}, \qquad (t,x)\in [0,T]\times [R,\infty)\times \real^{N-1},
$$
our aim is to find $R$, $A(R)$ and $B(R)$ to guarantee that
\begin{equation}\label{interm2}
\begin{split}
\mathcal{L}W_R & \ge 0 \ {\rm in} \ (0,T)\times(R,\infty)\times\real^{N-1}, \\ W_R(0) & \ge \|u_0\|_{\infty}\geq u_0 \ {\rm in} \ (R,\infty) \times \real^{N-1}.
\end{split}
\end{equation}
The latter is obviously fulfilled when
$$
W_R(0,x)\geq Z_R(0)^{1/m}=B(R)^{1/(1-q)}T^{1/(1-q)}\geq\|u_0\|_{\infty},
$$
which leads to a first condition on $B(R)$, namely
\begin{equation}\label{cond1}
B(R)\geq\frac{1}{T}\|u_0\|_{\infty}^{1-q}.
\end{equation}
We next verify the first condition in \eqref{interm2}. To this end, we notice that, since $m\geq1$,
$$
\big[ Y_R(x)+Z_R(t) \big]^{(m-1)/m} \geq Z_R(t)^{(m-1)/m} = \big[ B(R)(T-t) \big]^{(m-1)/(1-q)},
$$
hence
\begin{equation}\label{interm3}
\begin{split}
\partial_t W_R(t,x) & = - \frac{B(R)^{m/(1-q)}}{1-q} \frac{(T-t)^{(m+q-1)/(1-q)}}{\big[Y_R(x)+Z_R(t)\big]^{(m-1)/m}}\\
&\geq-\frac{B(R)^{1/(1-q)}}{1-q}(T-t)^{q/(1-q)}.
\end{split}
\end{equation}
Moreover,
\begin{equation}\label{interm4}
\begin{split}
\Delta W_R^m&=\Delta(Y_R(x)+Z_R(t))=\partial_{x_1}^2 Y_R(x)\\
&=\frac{2m(m+q)}{(m-q)^2}A(R)^{m/(m-q)}|x_1-2R|^{2q/(m-q)}.
\end{split}
\end{equation}
We then infer from \eqref{interm3} and \eqref{interm4} that, for $t\in (0,T)$ and $x\in (R,\infty)\times \real^{N-1}$,
\begin{equation*}
\begin{split}
\mathcal{L}W_R&\geq|x|^{\sigma}\left[A(R)^{m/(m-q)}|x_1-2R|^{2m/(m-q)}+B(R)^{m/(1-q)}(T-t)^{m/(1-q)}\right]^{q/m}\\
& \qquad - \frac{B(R)^{1/(1-q)}}{1-q} (T-t)^{q/(1-q)} - \frac{2m(m+q)}{(m-q)^2} A(R)^{m/(m-q)} |x_1-2R|^{2q/(m-q)}\\
&\geq \frac{x_1^{\sigma}}{2} B(R)^{q/(1-q)} (T-t)^{q/(1-q)} - \frac{B(R)^{1/(1-q)}}{1-q} (T-t)^{q/(1-q)}\\
& \qquad + \frac{x_1^{\sigma}}{2} A(R)^{q/(m-q)} |x_1-2R|^{2q/(m-q)} \\
& \qquad - \frac{2m(m+q)}{(m-q)^2} A(R)^{m/(m-q)} |x_1-2R|^{2q/(m-q)}\\
&\geq \frac{(T-t)^{q/(1-q)}}{2} B(R)^{q/(1-q)} \left[R^{\sigma} - \frac{2}{1-q} B(R) \right]\\
& \qquad + \frac{|x_1-2R|^{2q/(m-q)}}{2} A(R)^{q/(m-q)} \left[ R^{\sigma} - \frac{4m(m+q)}{(m-q)^2} A(R) \right]\geq0,
\end{split}
\end{equation*}
if the following conditions are satisfied
\begin{equation}\label{cond2}
A(R)\leq\frac{(m-q)^2}{4m(m+q)}R^{\sigma}, \qquad B(R)\leq\frac{1-q}{2}R^{\sigma}.
\end{equation}
Finally, the comparison on the parabolic lateral boundary $x\in\{R\}\times\real^{N-1}$ and $t\in[0,T]$ gives, together with \eqref{wp0},
$$
W_R(t,x)\geq Y_R(x)^{1/m}=A(R)^{1/(m-q)}R^{2/(m-q)}\geq\|u_0\|_{\infty}\geq u(t,x),
$$
provided
\begin{equation}\label{cond3}
A(R)R^2\geq\|u_0\|_{\infty}^{m-q}.
\end{equation}
Choosing thus
$$
A(R)=\frac{(m-q)^2}{4m(m+q)}R^{\sigma}, \qquad B(R)=\frac{1-q}{2}R^{\sigma}
$$
and
$$
R\geq R(T) := \max\left\{ \left[ \frac{2\|u_0\|_{\infty}^{1-q}}{(1-q)T} \right]^{1/\sigma} , \left[ \frac{4m(m+q)}{(m-q)^2}\|u_0\|_{\infty}^{m-q}\right]^{1/(\sigma+2)} \right\},
$$
we find that conditions \eqref{cond1}, \eqref{cond2} and \eqref{cond3} are satisfied and thus $W_R$ is a supersolution to Eq.~\eqref{eq1} in $(0,T)\times(R,\infty)\times\real^{N-1}$. The comparison principle entails then that
$$
W_R(t,x)\geq u(t,x), \qquad (t,x)\in[0,T]\times[R,\infty)\times\real^{N-1}.
$$
In particular, if $R\ge R(T)$ and $x\in\{2R\}\times\real^{N-1}$ we find that $u(T,x)\leq W_R(T,x)=0$, so that
$$
u(T)\equiv0, \qquad {\rm in} \quad (2R(T),\infty)\times\real^{N-1}.
$$
Due to the rotational invariance of Eq.~\eqref{eq1}, we have thus actually proved that $u(T)\equiv0$ in $\real^N\setminus B(0,2R(T))$. Since $T$ is arbitrary and $R(T)$ decreases with $T>0$, it follows that $u$ is compactly supported for positive times and, given $\tau>0$, its support remains localized in a fixed ball $B(0,2R(\tau))$ for any $t\in(\tau,\infty)$.
\end{proof}

\section{Existence and uniqueness of a self-similar solution}\label{sec.SSS}

Let us turn now our attention to non-negative self-similar solutions to Eq.~\eqref{eq1}, that is, non-negative solutions in the form~\eqref{SSS} with exponents $\alpha$ and $\beta$ given by \eqref{SSexp} and profile $f\ge 0$ solving the Cauchy problem~\eqref{SSODE}-\eqref{init.cond.ODE}. This rather long section is devoted to the properties of such profiles, proving that there exists a single one having an interface at a finite point. To fix the notation, by letting $F=f^m$ we can write the Cauchy problem~\eqref{SSODE}-\eqref{init.cond.ODE} in the following equivalent (but semilinear) form
\begin{equation}\label{ODE2}
F''(\xi)+\frac{N-1}{\xi}F'(\xi)+\alpha(F^{1/m}(\xi))-\beta\xi(F^{1/m})'(\xi)-\xi^{\sigma}F^{q/m}(\xi)=0,
\end{equation}
with initial condition
\begin{equation}\label{init.cond.ODE2}
F(0)=a^m>0, \qquad F'(0)=0.
\end{equation}
We begin our study of self-similar solutions with some basic properties of solutions to the Cauchy problem~\eqref{SSODE}-\eqref{init.cond.ODE}.

\subsection{Basic properties. Local behavior for $\xi$ small}\label{subsec.basic}

This section gathers some general properties of solutions to~\eqref{SSODE}-\eqref{init.cond.ODE} that will be used in the proofs of existence and uniqueness of the self-similar profile with interface. We include here, noticeably, the rather tedious but important local expansion of a solution to the Cauchy problem~\eqref{SSODE}-\eqref{init.cond.ODE} as $\xi\to 0$. Let $a\in(0,\infty)$. The Cauchy-Lipschitz theorem applied to the semilinear equivalent form~\eqref{ODE2} implies that the Cauchy problem \eqref{ODE2}-\eqref{init.cond.ODE2} has a unique positive solution $F(\cdot;a)\in C^2([0,\xi_{\rm max}(a))$, defined in a maximal interval $[0,\xi_{\rm max}(a))$, with the following alternative: either $\xi_{\rm max}(a)=\infty$ or
$$
\xi_{\rm max}(a)<\infty \;\;\text{ and }\;\; \lim\limits_{\xi\to\xi_{\rm max}(a)}\left[F(\xi;a)+\frac{1}{F(\xi;a)}\right]=\infty.
$$
We also set $f(\cdot;a) := F^{1/m}(\cdot;a)$ and define
\begin{equation}\label{support}
\xi_0(a):=\inf\{\xi\in(0,\xi_{\rm max}(a)):F(\xi;a)=0\}\in(0,\xi_{\rm max}(a)),
\end{equation}
so that $F(\xi;a)>0$ for any $\xi\in[0,\xi_0(a))$. Moreover, since $\sigma>0$, we deduce from Eq.~\eqref{ODE2} that
\begin{equation*}
F''(0;a)=-\frac{\alpha a}{N}<0, \qquad a\in(0,\infty),
\end{equation*}
and we can thus introduce the first zero of the derivative of $F(\cdot;a)$ (if any) and set
\begin{equation}\label{support.der}
\xi_1(a):=\sup\{\xi\in(0,\xi_0(a)):F'(\cdot;a)<0 \ {\rm on} \ (0,\xi)\}>0.
\end{equation}

\begin{lemma}\label{lem.min}
Let $a>0$. If $\xi_1(a)\in(0,\xi_0(a))$, then $F(\cdot;a)$ has a local strict minimum at $\xi=\xi_1(a)$ and
$$
\xi_1^{\sigma}(a) \geq \alpha f^{1-q}(\xi_1(a);a).
$$
\end{lemma}
\begin{proof}
Since $\xi_1(a)<\xi_0(a)$ and $F'(\xi;a)<0=F'(\xi_1(a);a)$ for any $\xi\in(0,\xi_1(a))$ (the latter follows from the definition \eqref{support.der} of $\xi_1(a)$), we conclude that $F''(\xi_1(a);a)\geq0$, and we infer by evaluating Eq. \eqref{ODE2} at $\xi=\xi_1(a)$ that
$$
0\leq F''(\xi_1(a);a)=f^{q}(\xi_1(a);a)\left[ \xi_1^{\sigma}(a) - \alpha f^{1-q}(\xi_1(a);a) \right].
$$
Either $\xi_1^{\sigma}(a)>\alpha f^{1-q}(\xi_1(a);a)$ and thus $F'(\cdot;a)>0$ in a right neighborhood of $\xi_1(a)$, which gives that $\xi_1(a)$ is a strict local minimum for $F(\cdot;a)$. Or $\xi_1^{\sigma}(a)=\alpha f^{1-q}(\xi_1(a);a)$ and thus $F''(\xi_1(a);a)=0$. Differentiating one more time Eq. \eqref{ODE2} with respect to $\xi$ and evaluating the resulting equation at $\xi=\xi_1(a)$, we get
\begin{equation*}
\begin{split}
F'''(\xi_1(a);a)&=-\frac{N-1}{\xi_1(a)}F''(\xi_1(a);a)+\frac{N-1}{\xi_1(a)^2}F'(\xi_1(a);a)\\&+(\beta-\alpha)f'(\xi_1(a);a)+\beta\xi_1(a)f''(\xi_1(a);a)\\
&+q\xi_1^{\sigma}(a) f^{q-1}(\xi_1(a);a) f'(\xi_1(a);a) + \sigma \xi_1^{\sigma-1}(a) f^q(\xi_1(a);a).
\end{split}
\end{equation*}
Noticing that
$$
f'(\xi_1(a);a)=\frac{1}{m}F^{(1-m)/m}(\xi_1(a);a) F'(\xi_1(a);a)=0
$$
and
\begin{align*}
f''(\xi_1(a);a) & = \frac{1}{m} F^{(1-m)/m}(\xi_1(a);a) F''(\xi_1(a);a) \\
&\qquad +\frac{1-m}{m^2} F^{(1-2m)/m}(\xi_1(a);a) |F'|^2(\xi_1(a);a) =0,
\end{align*}
we readily conclude that
$$
F'''(\xi_1(a);a)=\sigma\xi_1^{\sigma-1}(a) f^q(\xi_1(a);a)>0.
$$
Consequently, $F''(\cdot;a)>0$ in a right neighborhood of $\xi_1(a)$ and $F'(\cdot;a)$ also shares this property. It thus follows that $\xi_1(a)$ is a strict local minimum also in this case.
\end{proof}

We move now to the study of the local behavior of solutions $f(\cdot;a)$ in a small right neighborhood $\xi\in(0,\delta)$ of the origin. This is rather technical, but it will become a decisive argument in the proof of the uniqueness of the self-similar profile, in the same way as such a local asymptotic expansion has been used in previous works \cite{IL13, IMS22b}. The difficulty comes from the fact that the influence of the weight $|x|^{\sigma}$ in this expansion might come at very high order terms. The idea is then to prove that the local expansion of $f(\cdot;a)$ as $\xi\to0$ fully coincides with the one of the solution to the porous medium equation without any source, up to the (high) order where the influence of $|x|^{\sigma}$ appears, and give a precise form of the first coefficient where this influence is seen. To this end, let $a\in(0,\infty)$. Arguing as for~\eqref{ODE2}-\eqref{init.cond.ODE2}, there exists a unique positive solution $\Phi_a\in C^{\infty}([0,\xi_{\infty}(a)))$ of the Cauchy problem
\begin{equation}\label{ODE.PME}
\Phi_a''(\xi)+\frac{N-1}{\xi}\Phi_a'(\xi)+\alpha\phi_a(\xi)-\beta\xi\phi_a'(\xi)=0,
\end{equation}
with initial condition and notation
\begin{equation}\label{init.cond.ODE.PME}
\Phi_a(0)=a^m, \qquad \Phi_a'(0)=0, \qquad \phi_a:=\Phi_a^{1/m},
\end{equation}
defined on a maximal interval $[0,\xi_{\infty}(a))$. Introducing the rescaled function
$$
\chi(\xi) := a \phi_1(a^{-(m-1)/2}\xi), \qquad \xi\in \big[0, a^{(m-1)/2 \xi_\infty(1) } \big),
$$
we readily notice that $\chi$ also solves the Cauchy problem~\eqref{ODE.PME}-\eqref{init.cond.ODE.PME}. We infer from the well-posedness of~\eqref{ODE.PME}-\eqref{init.cond.ODE.PME} that $\xi_\infty(a) = a^{(m-1)/2} \xi_\infty(1)$ and
\begin{equation}\label{interm5}
\phi_a(\xi)=a\phi_1(a^{-(m-1)/2}\xi), \quad \Phi_a(\xi)=a^m\Phi_1(a^{-(m-1)/2}\xi), \qquad \xi\in [0,\xi_\infty(a)).
\end{equation}
Since $\Phi_1\in C^{\infty}([0,\xi_{\infty}(1))$ with $\Phi_1(0)=1$, the same holds true for $\phi_1=\Phi_1^{1/m}$, thus we can write, as $\xi\to0$, the Taylor expansions
\begin{equation}\label{Taylor.PME}
\phi_1(\xi)=\sum\limits_{j=0}^k\omega_j\xi^{j}+o(\xi^k), \qquad \Phi_1(\xi)=\sum\limits_{j=0}^k\Omega_j\xi^{j}+o(\xi^k)
\end{equation}
for any $k\in\mathbb{N}$, with
$$
\omega_j := \frac{1}{j!}\phi_1^{(j)}(0), \qquad \Omega_j := \frac{1}{j!} \Phi_1^{(j)}(0).
$$
Note that $\omega_1=\Omega_1=0$ by \eqref{init.cond.ODE.PME} while~\eqref{ODE.PME} and~\eqref{init.cond.ODE.PME} ensure that
\begin{equation}
	\omega_2 = - \frac{\alpha}{2mN}, \qquad \Omega_2 = - \frac{\alpha}{2N}. \label{z1}
\end{equation}
We then find by straightforward calculations that, for $k\geq2$ and as $\xi\to0$, the expansions~\eqref{Taylor.PME} up to orders $k+2$ and $k$ give
\begin{equation}\label{interm6}
\begin{split}
&\Phi_1'(\xi)=\sum\limits_{j=1}^{k+2}j\Omega_j\xi^{j-1}+o(\xi^{k+1}),\\
&\Phi_1''(\xi)=\sum\limits_{j=0}^{k}(j+2)(j+1)\Omega_{j+2}\xi^{j}+o(\xi^{k}),\\
&\frac{\Phi_1'(\xi)}{\xi}=\sum\limits_{j=0}^{k}(j+2)\Omega_{j+2}\xi^{j}+o(\xi^{k}),\\
&\xi\phi_1'(\xi)=\sum\limits_{j=1}^{k}j\omega_j\xi^{j}+o(\xi^{k})
\end{split}
\end{equation}
We deduce from gathering the expansions in \eqref{interm6} and inserting them into Eq.~\eqref{ODE.PME} that
$$
(j+2)(j+1)\Omega_{j+2}+(N-1)(j+2)\Omega_{j+2}+\alpha\omega_j-\beta j\omega_j=0,
$$
for any $1\leq j\leq k$, leading to
\begin{equation}\label{recurrence}
\Omega_{j+2}=\frac{j\beta-\alpha}{(j+2)(N+j)}\omega_j,
\end{equation}
valid for $1\leq j\leq k$. But \eqref{recurrence} is also valid for $j=0$ by~\eqref{z1}. Thus \eqref{recurrence} is valid for any integer $j\geq0$. We are now in a position to state our first lemma about the local expansion as $\xi\to0$ of the functions $F(\cdot;a)$ and $f(\cdot;a)$.

\begin{lemma}\label{lem.expPME}
Let $a>0$ and $F(\cdot;a)$, respectively $f(\cdot;a)$, be the solutions to the Cauchy problems \eqref{ODE2}-\eqref{init.cond.ODE2}, respectively \eqref{SSODE}-\eqref{init.cond.ODE}. Let $\Phi_a$ be the solution to the Cauchy problem \eqref{ODE.PME}-\eqref{init.cond.ODE.PME} and $\phi_a=\Phi_a^{1/m}$. Then, for any integer $k$ such that $2\leq k<2+\sigma$, we have
\begin{equation}\label{exp.PME}
F(\xi;a)-\Phi_a(\xi)=o(\xi^k), \qquad f(\xi;a)-\phi_a(\xi)=o(\xi^{k}), \qquad {\rm as} \ \xi\to0.
\end{equation}
\end{lemma}

\begin{proof}
We drop in the sequel for simplicity the explicit dependence on $a$ from the notation of $F$ and $f$, understanding that $a>0$ is fixed. Let us define $E=F-\Phi_a$ and observe that Eq.~\eqref{ODE2} ensures that:

$\bullet$ if $\sigma\not\in\mathbb{N}$, then $F\in C^{2+k_0,\sigma-k_0}([0,\xi_{\infty}(a)))$, where $k_0\in\mathbb{N}$ is the integer part of $\sigma$, and we have the following expansion as $\xi\to0$
\begin{equation}\label{exp.F1}
F(\xi)=\sum\limits_{j=0}^{k_0+2}B_j\xi^{j}+o(\xi^{k_0+2}), \qquad B_0=a^m, \ B_1=0, \ B_2=-\frac{\alpha a}{2N},
\end{equation}
and respectively
\begin{equation}\label{exp.f1}
f(\xi)=\sum\limits_{j=0}^{k_0+2}b_j\xi^{j}+o(\xi^{k_0+2}) \qquad {\rm as} \ \xi\to0.
\end{equation}

$\bullet$ if $\sigma\in\mathbb{N}$, then $F\in C^{\infty}([0,\xi_{\infty}(a)))$ and, setting $k_0=\sigma-1$, we get
\begin{equation}\label{exp.F2}
F(\xi)=\sum\limits_{j=0}^{k_0+3}B_j\xi^{j}+o(\xi^{k_0+3}), \qquad B_0=a^m, \ B_1=0, \ B_2=-\frac{\alpha a}{2N},
\end{equation}
and respectively
\begin{equation}\label{exp.f2}
f(\xi)=\sum\limits_{j=0}^{k_0+3}b_j\xi^{j}+o(\xi^{k_0+3}) \qquad {\rm as} \ \xi\to0.
\end{equation}

In both cases, we deduce from \eqref{z1}, \eqref{exp.F1}, and \eqref{exp.F2} that $E(\xi)=o(\xi^2)$, $E'(\xi)=o(\xi)$ as $\xi\to0$. We can now go on with an argument by induction and assume that $E'(\xi)=o(\xi^{k-1})$ for some integer $k\in[2,\sigma)$ and thus, also $E(\xi)=o(\xi^k)$ as $\xi\to0$. We infer from~\eqref{ODE2} and~\eqref{ODE.PME} that
\begin{equation}\label{interm7}
\begin{split}
E''(\xi)+\frac{N-1}{\xi}E'(\xi)&=-\alpha(F^{1/m}(\xi)-\Phi_a^{1/m}(\xi))\\
&+\beta\xi(F^{1/m}-\Phi_a^{1/m})'(\xi)+\xi^{\sigma}F^{q/m}(\xi).
\end{split}
\end{equation}
We then analyze one by one the orders of the terms in the right hand side of \eqref{interm7}, employing the induction hypothesis. We have
\begin{equation*}
\begin{split}
F^{1/m}(\xi)-\Phi_a^{1/m}(\xi)&=(E+\Phi_a)^{1/m}(\xi)-\Phi_a^{1/m}(\xi)\\
&=\Phi_a^{1/m}(\xi)\left[\left(1+\frac{E}{\Phi_a}\right)^{1/m}(\xi)-1\right]\\
&=\Phi_a^{1/m}(\xi)\left[\frac{E(\xi)}{m\Phi_a(\xi)}+o(E(\xi))\right]=o(\xi^k),
\end{split}
\end{equation*}
and also
\begin{equation*}
\begin{split}
\xi(F^{1/m}-\Phi_a^{1/m})'&(\xi)=\xi\left\{\Phi_a^{1/m}(\xi)\left[\left(1+\frac{E}{\Phi_a}\right)^{1/m}(\xi)-1\right]\right\}'\\
&=\xi\phi_a'(\xi)\left[\left(1+\frac{E}{\Phi_a}\right)^{1/m}(\xi)-1\right]\\
&+\frac{1}{m}\xi\phi_a(\xi)\left(1+\frac{E}{\Phi_a}\right)^{(1-m)/m}(\xi)\left[\frac{E'(\xi)}{\Phi_a(\xi)}-\frac{E(\xi)\Phi_a'(\xi)}{\Phi_a(\xi)^2}\right]\\
&=\xi\phi_a'(\xi)\left[\frac{E(\xi)}{m\Phi_a(\xi)}+o(E(\xi))\right]\\
&+\frac{\xi\phi_a(\xi)}{m\Phi_a(\xi)^2}\left(1+\frac{E(\xi)}{\Phi_a(\xi)}\right)^{(1-m)/m}(\Phi_aE'-\Phi_a'E)(\xi)\\
&=\xi\left[-\frac{a\alpha}{N}\xi+o(\xi)\right]\left[\frac{E(\xi)}{m\Phi_a(\xi)}+o(E(\xi))\right]\\
&+\frac{\xi}{m}\frac{\phi_a(\xi)}{\Phi_a^2(\xi)}\left(1+\frac{E(\xi)}{\Phi_a(\xi)}\right)^{(1-m)/m}\left[(a^m+o(\xi))E'(\xi)+\left(\frac{a\alpha}{N}\xi+o(\xi)\right)E(\xi)\right]\\
&=\xi^2E(\xi)+o(\xi^2E(\xi))+\frac{\xi}{ma^{2m-1}}\left[a^mE'(\xi)+\frac{a\alpha}{N}\xi E(\xi)+o(\xi^{k})\right]\\
&=o(\xi^k).
\end{split}
\end{equation*}
Since $k<\sigma$, we also have
$$
\xi^{\sigma}F^{q/m}(\xi)=o(\xi^k),
$$
and we infer from~\eqref{interm7} and the previous estimates that
$$
E''(\xi)+\frac{N-1}{\xi}E'(\xi)=o(\xi^k),
$$
or equivalently
$$
\frac{d}{d\xi}\left[\xi^{N-1}E'(\xi)\right]=o(\xi^{N+k-1}),
$$
which by integration leads to $E'(\xi)=o(\xi^{k+1})$ and thus $E(\xi)=o(\xi^{k+2})$ as $\xi\to0$. Thus the induction argument is complete, proving that
$$
E(\xi)=o(\xi^k), \qquad {\rm for \ any} \ k\in\mathbb{N}, \ 2\leq k<2+\sigma,
$$
which is equivalent to the first equality in \eqref{exp.PME}. Next, as $\xi\to0$ we also have
\begin{equation*}
\begin{split}
(f-\phi_a)(\xi)&=F^{1/m}(\xi)-\Phi_a^{1/m}(\xi)=(\Phi_a+E)^{1/m}(\xi)-\Phi_a^{1/m}(\xi)\\
&=\Phi_a(\xi)\left[\left(1+\frac{E(\xi)}{\Phi_a(\xi)}\right)^{1/m}-1\right]\\
&=\Phi_a(\xi)\left[\frac{E(\xi)}{m\Phi_a(\xi)}+o(E(\xi))\right]=o(\xi^{k}),
\end{split}
\end{equation*}
for any integer $k$ such that $2\leq k<2+\sigma$ and the proof is complete.
\end{proof}
In particular, recalling the Taylor expansions as $\xi\to0$ of $F(\xi)$ and $f(\xi)$ introduced in \eqref{exp.F1} and \eqref{exp.F2}, respectively \eqref{exp.f1} and \eqref{exp.f2}, we readily infer from~\eqref{interm5}, \eqref{Taylor.PME}, and Lemma~\ref{lem.expPME} that
\begin{equation}\label{interm8}
B_j=a^{m-j(m-1)/2}\Omega_j, \qquad b_j=a^{1-j(m-1)/2}\omega_j, \qquad 2\leq j\leq k_0+2,
\end{equation}
where $k_0$ is the largest integer strictly below $\sigma$ and $\Omega_j$, $\omega_j$ are the coefficients introduced in the expansions in \eqref{Taylor.PME}. Notice that all these coefficients do not depend on $\sigma$. We now proceed with the calculation of the next order of the expansion near $\xi=0$, where the influence of $\sigma$ comes into play.

\begin{lemma}\label{lem.expsigma}
Let $k_0$ be the largest integer strictly below $\sigma$. Then, as $\xi\to0$,
\begin{equation}\label{exp.smallF1}
F(\xi;a)=\sum\limits_{j=0}^{k_0+2}B_j\xi^j+\frac{a^q}{(\sigma+2)(\sigma+N)}\xi^{\sigma+2}+o(\xi^{\sigma+2}),
\end{equation}
if $\sigma\not\in\mathbb{N}$ and $k_0$ is the integer part of $\sigma$, or
\begin{equation}\label{exp.smallF2}
F(\xi;a)=\sum\limits_{j=0}^{k_0+3}B_j\xi^j+\frac{a^q}{(\sigma+2)(\sigma+N)}\xi^{\sigma+2}+o(\xi^{\sigma+2}),
\end{equation}
if $\sigma\in\mathbb{N}$ and $k_0=\sigma-1$.
\end{lemma}
\begin{proof}
We introduce
\begin{equation}\label{interm11}
H(\xi;a):=\xi^{N-1}F'(\xi;a)-\beta\xi^{N}f(\xi;a), \qquad \xi\in[0,\xi_{\infty}(a)).
\end{equation}
Dropping again the explicit dependence on $a>0$ in order to simplify the notation and differentiating with respect to $\xi$, we infer from \eqref{ODE2} that
\begin{equation}\label{interm9}
\begin{split}
H'(\xi)&=\xi^{N-1}\left[F''(\xi)+\frac{N-1}{\xi}F'(\xi)-\beta\xi f'(\xi)\right]-\beta N\xi^{N-1}f(\xi)\\
&=-(\alpha+N\beta)\xi^{N-1}f(\xi)+\xi^{\sigma+N-1}f^q(\xi).
\end{split}
\end{equation}
We now split the analysis according to whether $\sigma>0$ is an integer or not.

\medskip

\noindent \textbf{Case~1. $\sigma\not\in\mathbb{N}$}. We deduce from~\eqref{exp.f1} and~\eqref{interm9} that
\begin{equation*}
\begin{split}
H'(\xi)&=-(\alpha+N\beta)\sum\limits_{j=0}^{k_0+2}b_j\xi^{N+j-1}+o(\xi^{N+k_0+1})+\xi^{N+\sigma-1}(b_0+b_2\xi^2+o(\xi^2))^q\\
&=-(\alpha+N\beta)\sum\limits_{j=0}^{k_0+2}b_j\xi^{N+j-1}+o(\xi^{N+k_0+1})+b_0^q\xi^{N+\sigma-1}\left[1+\frac{b_2}{b_0}\xi^2+o(\xi^2)\right]^q\\
&=-(\alpha+N\beta)\sum\limits_{j=0}^{k_0+2}b_j\xi^{N+j-1}+o(\xi^{N+k_0+1})+b_0^q\xi^{N+\sigma-1}+o(\xi^{N+\sigma})\\
&=-(\alpha+N\beta)\sum\limits_{j=0}^{k_0}b_j\xi^{N+j-1}+b_0^q\xi^{N+\sigma-1}+o(\xi^{N+\sigma-1})
\end{split}
\end{equation*}
and we further get by integration that
$$
H(\xi)=-(\alpha+N\beta)\sum\limits_{j=0}^{k_0}\frac{b_j}{N+j}\xi^{N+j}+\frac{b_0^q}{N+\sigma}\xi^{N+\sigma}+o(\xi^{N+\sigma}).
$$
Recalling the definition of $H(\xi)$, we obtain
$$
\xi^{N-1}F'(\xi)=\beta\xi^Nf(\xi)-(\alpha+N\beta)\sum\limits_{j=0}^{k_0}\frac{b_j}{N+j}\xi^{N+j}+\frac{b_0^q}{N+\sigma}\xi^{N+\sigma}+o(\xi^{N+\sigma})
$$
or equivalently, thanks to~\eqref{exp.f1},
$$
F'(\xi)=\sum\limits_{j=0}^{k_0}\left(\beta-\frac{\alpha+N\beta}{N+j}\right)b_j\xi^{j+1}+\frac{b_0^q}{N+\sigma}\xi^{\sigma+1}+o(\xi^{\sigma+1}),
$$
which leads to \eqref{exp.smallF1} by one more step of integration and taking into account the relations between the coefficients $B_j$, $b_j$, $\Omega_j$ and $\omega_j$ obtained by gathering~\eqref{recurrence} and~\eqref{interm8}.

\medskip

\noindent \textbf{Case 2. $\sigma\in\mathbb{N}$}. We proceed in a similar manner as in Case~1 and infer from~\eqref{exp.f2} and~\eqref{interm9} that
\begin{equation*}
\begin{split}
H'(\xi)&=-(\alpha+N\beta)\sum\limits_{j=0}^{k_0+3}b_j\xi^{N+j-1}+o(\xi^{N+k_0+2})+b_0^q\xi^{k_0+N}\left[1+\frac{b_2b_0^{q-1}}{q}\xi^2+o(\xi^2)\right]\\
&=-(\alpha+N\beta)\sum\limits_{j=0}^{k_0+1}b_j\xi^{N+j-1}+b_0^q\xi^{k_0+N}+o(\xi^{k_0+N}),
\end{split}
\end{equation*}
hence by integrating with respect to $\xi$ and separating $F'(\xi)$ from the definition of $H(\xi)$, we obtain
$$
F'(\xi)=\sum\limits_{j=0}^{k_0+1}\frac{j\beta-\alpha}{N+j}b_j\xi^{j+1}+\frac{b_0^q}{N+k_0+1}\xi^{k_0+2}+o(\xi^{k_0+2}).
$$
We get~\eqref{exp.smallF2} by integrating once more with respect to $\xi$, using the relations between the coefficients $B_j$, $b_j$, $\Omega_j$ and $\omega_j$ obtained by gathering~\eqref{recurrence} and~\eqref{interm8} and recalling that $\sigma=k_0+1$ in this case.
\end{proof}
We are now ready to proceed to the proof of existence of a self-similar solution to Eq.~\eqref{eq1}.

\subsection{Existence of a self-similar solution}\label{subsec.exist}

Recalling the definitions~\eqref{support} and~\eqref{support.der} of $\xi_0(a)$ and $\xi_1(a)$, we introduce the following three sets
\begin{equation*}
\begin{split}
&\mathcal{A}:=\{a>0:\xi_0(a)<\infty \ {\rm and} \ (f^m)'(\xi;a)<0, \ {\rm for} \ \xi\in(0,\xi_0(a)]\}\\
&\mathcal{C}:=\{a>0:\xi_1(a)<\xi_0(a)\},\\
&\mathcal{B}:=(0,\infty)\setminus(\mathcal{A}\cup\mathcal{C}).
\end{split}
\end{equation*}
We prove first that the sets $\mathcal{A}$ and $\mathcal{C}$ are both non-empty and open. To this end, let us introduce $g(\cdot;a)$ defined by the following rescaling
\begin{equation}\label{exist.resc}
f(\xi;a)=ag(a^{\gamma}\xi;a), \qquad \eta=a^{\gamma}\xi,
\end{equation}
where $\gamma$ will be fixed later. We obtain by plugging the rescaling~\eqref{exist.resc} into~\eqref{ODE2} that $g(\cdot;a)$ solves the ordinary differential equation
\begin{subequations}\label{ODE.resc}
\begin{equation} \label{ODE.resc1}
\begin{split}
(g^m)''(\eta)+\frac{N-1}{\eta}(g^m)'(\eta)& +a^{1-m-2\gamma}(\alpha g(\eta)-\beta\eta g'(\eta)) \\
& \qquad -a^{q-m-\gamma(\sigma+2)}\eta^{\sigma}g^q(\eta)=0,
\end{split}
\end{equation}
with initial conditions
\begin{equation}\label{ODE.resc2}
g(0)=1, \quad g'(0)=0.
\end{equation}
\end{subequations}
Here and in the sequel we drop the explicit dependence on $a>0$ from the notation $g(\cdot;a)$ for simplicity.

\begin{lemma}\label{lem.A}
The set $\mathcal{A}$ is open. Moreover, there exists $a_*>0$ such that $(0,a_*)\subseteq\mathcal{A}$.
\end{lemma}

\begin{proof}
We choose $\gamma=-(m-1)/2$ in \eqref{exist.resc} and we note that
$$
q-m-\gamma(\sigma+2)=\frac{\sigma(m-1)+2(q-1)}{2}>0
$$
in our range of exponents~\eqref{range.exp}, so that Eq.~\eqref{ODE.resc} reduces in the limit $a\to0$ to
\begin{equation}\label{ODE.setA}
\begin{split}
& (h^m)''(\eta)+\frac{N-1}{\eta}(h^m)'(\eta)+\alpha h(\eta)-\beta\eta h'(\eta)=0, \\
& \qquad h(0)=1, \quad h'(0) = 0.
\end{split}
\end{equation}
We claim that there exists $\eta_0\in(0,\infty)$ such that
$$
h(\eta_0)=0 \;\;\text{ and }\;\; (h^m)'(\eta_0)<0.
$$
To prove the claim, we argue as in the proof of \cite[Theorem~2]{Shi04}. Set
$$
\eta_0:=\inf\{\eta>0: h(\eta)=0\}>0, \qquad \eta_1:=\inf\{\eta\in(0,\eta_0): (h^m)'(\eta)=0\},
$$
the latter being well defined and positive due to $(h^m)''(0)=-\alpha/N<0$.

Assume first for contradiction that $\eta_1<\eta_0$. Then $(h^m)''(\eta_1)\geq0$, while \eqref{ODE.setA} evaluated at $\eta=\eta_1$ gives $(h^m)''(\eta_1)=-\alpha h(\eta_1)<0$ and a contradiction. Therefore, $\eta_1=\eta_0$.

Assume now for contradiction that $\eta_0=\infty$ and thus $(h^m)'(\eta)<0$ and $h'(\eta)<0$ for any $\eta>0$. It follows from Eq.~\eqref{ODE.setA} that
\begin{equation*}
\begin{split}
\frac{d}{d\eta}\left[\eta^{N-1}(h^m)'(\eta)+\frac{\alpha}{N}\eta^Nh(\eta)\right]&=\eta^{N-1}\left[\beta\eta h'(\eta)-\alpha h(\eta)\right]+\alpha\eta^{N-1}h(\eta)+\frac{\alpha}{N}\eta^Nh'(\eta)\\
&=\left(\beta+\frac{\alpha}{N}\right)\eta^Nh'(\eta)<0,
\end{split}
\end{equation*}
which after integration gives
\begin{equation*}
\begin{split}
\eta^{N-1}(h^m)'(\eta)+\frac{\alpha}{N}\eta^Nh(\eta)&\leq\left(\beta+\frac{\alpha}{N}\right)\int_0^{\eta}\tau^Nh'(\tau)\,d\tau\\
&\leq-\delta:=\left(\beta+\frac{\alpha}{N}\right)\int_0^{1}\tau^Nh'(\tau)\,d\tau,
\end{split}
\end{equation*}
provided $\eta>1$, since $h'(\eta)<0$ for any $\eta>0$. We further infer that
\begin{equation}\label{interm10}
(h^m)'(\eta)+\frac{\alpha}{N}\eta h(\eta)+\delta\eta^{1-N}\leq0, \qquad \eta\geq1.
\end{equation}
We next apply Young's inequality
$$
\left(\frac{\alpha}{N}\eta h(\eta)\right)^{(N-1)/N}(\delta\eta^{1-N})^{1/N}\leq\frac{\alpha}{N}\eta h(\eta)+\delta\eta^{1-N},
$$
to deduce from~\eqref{interm10} that
$$
\left(\frac{\alpha}{N}\right)^{(N-1)/N}\delta^{1/N}h^{(N-1)/N}(\eta)\leq-(h^m)'(\eta)=-mh^{m-1}(\eta)h'(\eta),
$$
or equivalently
$$
h^{(N(m-2)+1)/N}(\eta)h'(\eta)+\left(\frac{\alpha}{N}\right)^{(N-1)/N}\frac{\delta^{1/N}}{m}\leq0,
$$
for any $\eta\geq1$. Consequently, since $m+1/N>1$, there exists $c_1>0$ such that
$$
\left(h^{(N(m-1)+1)/N}(\eta)\right)'\leq-c_1, \qquad \eta\geq1,
$$
which gives after integration that
$$
h^{(N(m-1)+1)/N}(\eta)\leq h^{(N(m-1)+1)/N}(1)-c_1(\eta-1),
$$
leading to a contradiction with the positivity of $h$ for $\eta>1$ large enough. Therefore $\eta_0<\infty$. We are left only to show that $(h^m)'(\eta_0)<0$. To this end, we notice that Eq.~\eqref{ODE.setA} implies
$$
\frac{d}{d\eta}\left[\eta^{N-1}(h^m)'(\eta)\right]=-\eta^{N-1}\left[\alpha h(\eta)-\beta\eta h'(\eta)\right]<0,
$$
for any $\eta\in(0,\eta_0)$. Therefore,
$$
\eta^{N-1}(h^m)'(\eta)<\left(\frac{\eta_0}{2}\right)^{N-1}(h^m)'\left(\frac{\eta_0}{2}\right)<0, \qquad \eta\in\left(\frac{\eta_0}{2},\eta_0\right),
$$
and passing to the limit as $\eta\to\eta_0$ from the left, we infer that
$$
\eta_0^{N-1}(h^m)'(\eta_0)<0,
$$
as claimed. Standard continuity arguments in \eqref{ODE.resc} with respect to the parameter $a>0$ end up the proof.
\end{proof}

We now characterize in a similar manner the set $\mathcal{C}$.

\begin{lemma}\label{lem.C}
The set $\mathcal{C}$ is open. Moreover, there exists $a^{*}>0$ such that $(a^{*},\infty)\subseteq\mathcal{C}$.
\end{lemma}

\begin{proof}
We now choose $\gamma=(q-m)/(\sigma+2)$ in \eqref{exist.resc} and we observe that
$$
1-m-2\gamma=-\frac{\sigma(m-1)+2(q-1)}{\sigma+2}<0
$$
in our range of exponents~\eqref{range.exp}. Thus, in the limit $a\to\infty$, Eq.~\eqref{ODE.resc} reduces to
\begin{equation}\label{ODE.setC}
\begin{split}
& (l^m)''(\eta)+\frac{N-1}{\eta}(l^m)'(\eta)-\eta^{\sigma}l^q(\eta)=0, \\
& \qquad l(0)=1, \quad l'(0)= 0
\end{split}
\end{equation}
which has a unique positive solution $l^m\in C^2([0,\eta_\infty))$ defined on a maximal existence interval $[0,\eta_\infty)$ with the alternative: either $\eta_\infty=\infty$ or
$$
\eta_\infty < \infty \;\;\text{ and }\;\; \lim_{\eta\to \eta_\infty} \left( l(\eta) + \frac{1}{l(\eta)} \right) = \infty.
$$

Assume for contradiction that
$$
\eta_0 := \inf\big\{ \eta\in (0,\eta_\infty)\ :\ l(\eta)=0 \big\} < \eta_\infty,
$$
so that $l'(\eta_0)\le 0$. We readily notice from~\eqref{ODE.setC} that
\begin{equation}\label{z2}
\frac{d}{d\eta}\left(\eta^{N-1}(l^m)'(\eta)\right)=\eta^{N+\sigma-1}l^q(\eta)>0, \qquad \eta\in (0,\eta_0),
\end{equation}
and thus
$$
\eta^{N-1}(l^m)'(\eta)>0, \qquad \eta\in(0,\eta_0),
$$
and a contradiction. Consequently, $\eta_0=\eta_\infty$ and we infer from~\eqref{z2} after integration that
$$
0 < \eta^{N-1} (l^m)'(\eta) = \int_0^\eta \eta_*^{N+\sigma-1} l^q(\eta_*)\,d\eta_*, \qquad \eta\in (0,\eta_\infty).
$$
In particular, the function $l$ is increasing on $(0,\eta_\infty)$ and this property implies that
$$
\eta^{N-1} (l^m)'(\eta) \le l^q(\eta) \int_0^\eta \eta_*^{N+\sigma-1}\,d\eta_* = \frac{\eta^{N+\sigma}}{N+\sigma} l^q(\eta), \qquad \eta\in (0,\eta_\infty).
$$
Owing to the positivity of both $l$ and $\eta$ on $[0,\eta_\infty)$, we further obtain
$$
(l^{m-q})'(\eta) \le \frac{m-q}{m} \frac{\eta^{\sigma+1}}{N+\sigma}, \qquad \eta\in (0,\eta_\infty).
$$
Integrating once more and recalling that $l$ is increasing on $(0,\eta_\infty)$, we conclude that
$$
1 \le l^{m-q}(\eta) \le 1 + \frac{m-q}{m} \frac{\eta^{\sigma+2}}{(\sigma+2)(N+\sigma)}, \qquad \eta\in (0,\eta_\infty),
$$
a property which discards the possible finiteness of $\eta_\infty$ according to the alternative above. Consequently, the function $l$ is increasing on $(0,\infty)$ and it has a strict minimum at $\eta=0$. Coming back to the solution $g(\cdot;a)$ to Eq.~\eqref{ODE.resc}, since $l'(1)>0$, there exists $a^{*}>0$ such that
$$
\max_{[0,1]}|g(\cdot;a) - l| \le \frac{1}{2} \;\;\text{ and }\;\; |g'(1;a)-l'(1)|\leq\frac{l'(1)}{2}, \qquad {\rm for \ any} \ a\in(a^{*},\infty);
$$
that is, recalling that $l\ge 1$ on $[0,1]$, $g(\eta;a)\ge 1/2$ for $\eta\in [0,1]$ and $g'(1;a)\geq l'(1)/2>0$. Since we already know that $(g^m;a)'(\eta)$ is negative in a right neighborhood of $\eta=0$, there is at least one positive local strict minimum to $g(\cdot;a)$ in $[0,1]$ and thus $a\in\mathcal{C}$ for any $a\in(a^{*},\infty)$. A continuity argument gives that $\mathcal{C}$ is open.
\end{proof}

We infer from Lemmas~\ref{lem.A} and~\ref{lem.C} that the complementary set $\mathcal{B}$ is non-empty and closed. We then characterize the elements of the set $\mathcal{B}$ as follows.

\begin{lemma}\label{lem.B}
Let $a\in\mathcal{B}$. Then $\xi_0(a)<\infty$ and $(f^m)'(\xi_0(a))=0$.
\end{lemma}

\begin{proof}
Assume for contradiction that $\xi_0(a)=\infty$. Then, since $a\not\in\mathcal{C}$, it follows from Lemma~\ref{lem.min} that $\xi_1(a)=\infty$ and that $f(\cdot;a)$ is decreasing on $(0,\infty)$ and positive, thus $0<f(\xi;a)<a$ for any $\xi>0$. Consequently, recalling~\eqref{SSS},
$$
(t,x)\mapsto (1+t)^{-\alpha} f(|x|(1+t)^\beta;a)
$$
is a weak solution to \eqref{eq1} with a non-negative and bounded initial condition and its support at any time is the whole space $\real^N$. But this contradicts the instantaneous shrinking of the supports of solutions proved in Theorem~\ref{th.wp}. Thus, $\xi_0(a)<\infty$ and, since $f^m$ is decreasing on $(0,\xi_0(a))$ and $a\not\in\mathcal{A}$, we conclude that $(f^m)'(\xi_0(a))=0$.
\end{proof}

Summarizing the outcome of the analysis performed so far, elements in the set $\mathcal{B}$ correspond to self-similar profiles having an interface at $\xi=\xi_0(a)<\infty$ giving rise to a self-similar solution to Eq.~\eqref{eq1}, thus completing the existence part of Theorem~\ref{th.uniqSS}.

\subsection{Interface behavior when $a\in\mathcal{B}$}\label{subsec.interface}

Once proved that there are elements in the set $\mathcal{B}$ and that they are compactly supported with a smooth contact at the edge of the support $\xi_0(a)\in(0,\infty)$, we establish now the local behavior of profiles $f(\cdot;a)$ for $a\in\mathcal{B}$ as $\xi\to\xi_0(a)$. We start with a general upper bound. As before, we drop the explicit dependence on $a$ from the notation.

\begin{lemma}\label{lem.i1}
Assume that $a\in\mathcal{B}$ and $f=f(\cdot;a)$. Then
\begin{equation}\label{i1}
|(f^{m-q})'(\xi)|\leq 2^{N-1}\xi_0^{\sigma}(\xi_0-\xi), \qquad \xi\in\left(\frac{\xi_0}{2},\xi_0\right).
\end{equation}
Moreover, there exist $C_1>0$ and $C_2>0$ depending only on $m$, $N$, $q$, $\sigma$ and $\xi_0$ such that the following upper bounds for $f$ hold true:
\begin{equation}\label{i2}
f(\xi)\leq C_1(\xi_0-\xi)^{2/(m-q)}, \qquad \xi\in\left(\frac{\xi_0}{2},\xi_0\right),
\end{equation}
and
\begin{equation}\label{i3}
f(\xi)\leq C_2(\xi_0-\xi)^{1/(1-q)}, \qquad \xi\in\left(\frac{\xi_0}{2},\xi_0\right).
\end{equation}
\end{lemma}

Notice that \eqref{i2} is a better upper bound than \eqref{i3} if and only if $m+q<2$.

\begin{proof}
Recalling the definition~\eqref{interm11} of $H(\xi)$ and \eqref{interm9}, we have for $\xi\in(0,\xi_0)$
$$
H'(\xi)=\xi^{N+\sigma-1}f^q(\xi) - (\alpha+N\beta) \xi^{N-1} f(\xi) \leq \xi^{N+\sigma-1} f^q(\xi).
$$
Integrating this inequality over $(\xi,\xi_0)$ and taking into account that $f(\xi_0)=0$ and $F'(\xi_0)=0$, we obtain
$$
-\xi^{N-1}F'(\xi)+\beta\xi^N f(\xi)\leq\int_{\xi}^{\xi_0}\eta^{N+\sigma-1}f^q(\eta)\,d\eta.
$$
Using the fact that $f$ is decreasing since $a\in\mathcal{B}$ (see Lemma~\ref{lem.B}) and the definition of $F$ we further find
\begin{equation*}
\begin{split}
m\xi^{N-1}f^{m-1}(\xi)|f'(\xi)|+\beta\xi^Nf(\xi)&\leq f^q(\xi)\int_{\xi}^{\xi_0}\eta^{N+\sigma-1}\,d\eta\\
&\leq f^q(\xi)\xi_0^{N+\sigma-1}(\xi_0-\xi),
\end{split}
\end{equation*}
thus, taking into account the positivity of $f(\xi)$ for $\xi\in(0,\xi_0)$, we get
\begin{equation}\label{interm12}
m\xi^{N-1}f^{m-1-q}(\xi)|f'(\xi)|+\beta\xi^Nf^{1-q}(\xi)\leq\xi_0^{N+\sigma-1}(\xi_0-\xi),
\end{equation}
for any $\xi\in(0,\xi_0)$. Consider now $\xi\in(\xi_0/2,\xi_0)$. On the one hand, using the first term in the left-hand side of \eqref{interm12}, we infer that
\begin{equation}\label{interm13}
\frac{m}{m-q}\left(\frac{\xi_0}{2}\right)^{N-1}|(f^{m-q})'(\xi)|\leq\xi_0^{N+\sigma-1}(\xi_0-\xi),
\end{equation}
and \eqref{i1} follows readily from~\eqref{interm13}. Recalling that $f'<0$ on $(0,\xi_0)$ and integrating~\eqref{interm13} over $(\xi,\xi_0)$ give
$$
f^{m-q}(\xi)\leq\frac{2^{N-2}(m-q)}{m}\xi_0^{\sigma}(\xi_0-\xi)^2,
$$
and we have shown~\eqref{i2}. On the other hand, using the second term in the left-hand side of \eqref{interm12}, we obtain, for $\xi\in (\xi_0/2,\xi_0)$,
$$
\beta\xi_0^{N}f^{1-q}(\xi)\leq 2^N\xi_0^{N+\sigma-1}(\xi_0-\xi),
$$
from which the last upper bound~\eqref{i3} follows.
\end{proof}

The next two technical results establish properties of the derivatives of some powers of $f(\cdot;a)$ with $a\in\mathcal{B}$ near the interface point and will be very useful in establishing the precise local behavior as $\xi\to\xi_0$.

\begin{lemma}\label{lem.i2}
Let $a\in\mathcal{B}$. Then
$$
\limsup\limits_{\xi\to\xi_0}\left(f^{(m-q)/2}\right)'(\xi;a)>-\infty.
$$
\end{lemma}

\begin{proof}
We argue by contradiction and assume that the limit in the statement of Lemma~\ref{lem.i2} is equal to $-\infty$. Let $R>C_1^{(m-q)/2}$, where $C_1>0$ is the constant in the upper bound \eqref{i2}. It follows from the definition of the limit that there exists $\xi_R\in(\xi_0/2,\xi_0)$ such that
$$
\left(f^{(m-q)/2}\right)'(\xi)\leq-R, \qquad \xi\in(\xi_R,\xi_0).
$$
Taking $\xi\in(\xi_R,\xi_0)$ and integrating over $(\xi,\xi_0)$ easily leads to
$$
f(\xi)\geq R^{2/(m-q)}(\xi_0-\xi)^{2/(m-q)}>C_1(\xi_0-\xi)^{2/(m-q)},
$$
which contradicts \eqref{i2}.
\end{proof}

\begin{lemma}\label{lem.i3}
Let $a\in\mathcal{B}$ and assume that $m+q>2$. Then
$$
\limsup\limits_{\xi\to\xi_0}\left(f^{m-1}\right)'(\xi;a)=0.
$$
\end{lemma}

\begin{proof}
We argue by contradiction and assume that there exist $\epsilon>0$ and $\xi_{\epsilon}\in(\xi_0/2,\xi_0)$ such that
$$
(f^{m-1})'(\xi)\leq-\epsilon, \qquad \xi\in(\xi_{\epsilon},\xi_0).
$$
Integrating this inequality over $(\xi,\xi_0)$ (for $\xi\in(\xi_{\epsilon},\xi_0)$) gives
\begin{equation}\label{interm14}
f^{m-1}(\xi)\geq\epsilon(\xi_0-\xi), \qquad \xi\in(\xi_{\epsilon},\xi_0).
\end{equation}
We thus deduce from \eqref{i3} and \eqref{interm14} that
$$
\epsilon\leq\frac{f^{m-1}(\xi)}{\xi_0-\xi}\leq C_2^{m-1}(\xi_0-\xi)^{(m+q-2)/(1-q)}, \qquad \xi\in(\xi_{\epsilon},\xi_0),
$$
which leads to a contradiction if we pass to the limit as $\xi\to\xi_0$, taking into account that $m+q-2>0$.
\end{proof}

We are now ready to complete the local analysis near the interface point by computing the precise behavior of $f(\xi;a)$ as $\xi\to\xi_0(a)$ when $a\in\mathcal{B}$. To this end, we will employ ``surgically'' a phase space technique in order to identify the profiles solving \eqref{SSODE} with orbits entering some specific critical points of a quadratic, autonomous dynamical system. We will make a strong use of Lemmas~\ref{lem.i2} and~\ref{lem.i3} to select the right critical point. The analysis will be split into two cases, according to the sign of $m+q-2$, and we start with the range $m+q\in (1,2]$.

\medskip

\begin{proposition}\label{prop.interf.low}
Let $a\in\mathcal{B}$ and assume that $m+q\leq2$. Then, introducing
$$
K_1 := \left( \frac{m-q}{\sqrt{2m(m+q)}} \right)^{2/(m-q)}
$$
and
$$
K_2(z) := \left( \sqrt{1 + \frac{\beta^2 z^2}{2m(m+q)}} - \frac{\beta z}{\sqrt{2m(m+q)}} \right)^{2/(m-q)},
$$
the behavior of $f(\xi;a)$ as $\xi\to \x_0(a)$ is given by (with $\xi_0=\xi_0(a)$)
$$
f(\xi;a)=K_1\xi_0^{\sigma/(m-q)}(\xi_0-\xi)^{2/(m-q)}+o\Big((\xi_0-\xi)^{2/(m-q)}\Big)
$$
when $m+q\in (1,2)$ and
\begin{align*}
	f(\xi;a) & = K_1\xi_0^{\sigma/(m-q)} K_2\left(\xi_0^{(2-\sigma)/2} \right) (\xi_0-\xi)^{2/(m-q)} \\
	& \qquad + o\Big((\xi_0-\xi)^{2/(m-q)}\Big)
\end{align*}
when $m+q=2$.
\end{proposition}

\begin{proof}
Consider $a\in\mathcal{B}$. Setting $f=f(\cdot;a)$, we introduce the following change of unknown functions
\begin{equation}\label{change.small}
\begin{split}
\mathcal{X}(\xi) & := \sqrt{m} \xi^{-(\sigma+2)/2} f^{(m-q)/2}(\xi), \\
\mathcal{Y}(\xi) & := \sqrt{m} \xi^{-\sigma/2} f^{(m-q-2)/2}(\xi) f'(\xi), \\
\mathcal{Z}(\xi) & := \frac{\alpha}{\sqrt{m}} \xi^{(2-\sigma)/2} f^{(2-m-q)/2}(\xi),
\end{split}
\end{equation}
as well as a new independent variable $\eta$ via the integral representation
\begin{equation}\label{PSvar.low}
\eta(\xi) := \frac{1}{\sqrt{m}} \int_0^\xi f^{(q-m)/2}(\xi_*) \xi_*^{\sigma/2} \,d\xi_*, \qquad \xi\in [0,\xi_0].
\end{equation}
Since $f$ is positive in $[0,\xi_0)$ as $a\in\mathcal{B}$ and
$$
f^{(q-m)/2}(\xi_*) \ge \frac{C_1^{(q-m)/2}}{\xi_0-\xi_*}, \qquad \xi_*\in (0,\xi_0),
$$
by~\eqref{i2}, we see that $\eta$ is a diffeomorphism from $[0,\xi_0)$ onto $[0,\infty)$. Also, due to the definition of $\mathcal{Y}$,
\begin{equation}\label{interm15}
\begin{split}
f'(\xi) & = \frac{1}{\sqrt{m}}\xi^{\sigma/2} f^{(q+2-m)/2}(\xi) \mathcal{Y}(\xi), \\
f''(\xi) & = \frac{\sigma}{2\sqrt{m}} \xi^{(\sigma-2)/2}f^{(q+2-m)/2}(\xi) \mathcal{Y}(\xi) + \frac{1}{\sqrt{m}} \xi^{\sigma/2} f^{(q+2-m)/2}(\xi) \mathcal{Y}'(\xi)\\
& \qquad + \frac{q+2-m}{2\sqrt{m}} \xi^{\sigma/2} f^{(q-m)/2}(\xi) f'(\xi) \mathcal{Y}(\xi)\\
& = \frac{\sigma}{2\sqrt{m}} \xi^{(\sigma-2)/2} f^{(q+2-m)/2}(\xi) \mathcal{Y}(\xi) + \frac{1}{\sqrt{m}} \xi^{\sigma/2} f^{(q+2-m)/2}(\xi) \mathcal{Y}'(\xi)\\
& \qquad +\frac{q+2-m}{2m} \xi^{\sigma} f^{q+1-m}(\xi) \mathcal{Y}^2(\xi).
\end{split}
\end{equation}
Replacing the derivatives obtained in \eqref{interm15} into Eq.~\eqref{ODE2} and changing at the end to
\begin{equation}\label{change.small1}
	(\mathcal{X},\mathcal{Y},\mathcal{Z}) := (X\circ\eta, Y\circ\eta, Z\circ\eta)
\end{equation}
leads, after technical but straightforward calculations, to the following autonomous dynamical system
\begin{equation}\label{PSSyst.low}
\left\{
\begin{split}
	\dot{X}&= X \left[ \frac{m-q}{2} Y - \frac{\sigma+2}{2} X\right]\\
	\dot{Y}&= - \frac{m+q}{2} Y^2 - \left( N-1+\frac{\sigma}{2} \right) XY - XZ + \frac{\beta}{\alpha} YZ + 1\\
	\dot{Z}&= Z \left[ \frac{2-m-q}{2}Y + \frac{2-\sigma}{2}X \right],
\end{split}
\right.
\end{equation}
for $\eta\in (0,\infty)$, where the dot denotes the derivative taken with respect to the variable $\eta$ introduced in~\eqref{PSvar.low}. Observe that, due to this change of variable, the behavior of $f(\xi)$ as $\xi\to\xi_0$ is encoded in that of $(X,Y,Z)(\eta)$ as $\eta\to\infty$.

Now, since $a\in\mathcal{B}$, the profile $f$ is positive and decreasing in $(0,\xi_0)$ which implies, along with Lemma~\ref{lem.i2}, \eqref{change.small}, and \eqref{change.small1} that
\begin{equation}\label{ps1}
	Y < 0 \;\;\text{ in }\;\; (0,\infty), \qquad \limsup_{\eta\to\infty} Y(\eta) > -\infty
\end{equation}
and
\begin{equation}\label{ps3}
	X>0, \quad Z>0 \;\;\text{ in }\;\; (0,\infty).
\end{equation}
Moreover, the condition $m+q\le 2$ guarantees that $\sigma>2$ and we infer from \eqref{change.small}, \eqref{change.small1}, \eqref{PSSyst.low}, \eqref{ps1}, and \eqref{ps3} that
\begin{equation}\label{ps2}
	\dot{X}\le 0, \quad \dot{Z} \le 0 \;\;\text{ in }\;\; (0,\infty),
\end{equation}
with
\begin{equation}\label{ps4}
	\lim_{\eta\to\infty} X(\eta) = 0, \quad \lim_{\eta\to\infty} Z(\eta) = Z_\star,
\end{equation}
where $Z_\star := 0$ when $m+q<2$ and $Z_\star := \alpha \xi_0^{(2-\sigma)/2}/\sqrt{m}$ when $m+q=2$.

We next exploit the analysis performed so far to identify the behavior of $Y(\eta)$ as $\eta\to\infty$. We begin with the behavior of $Y$ along unbounded sequences of its critical points and claim that, if $(\eta_j)_{j\ge 1}$ is a non-decreasing sequence satisfying
\begin{subequations}\label{ps100}
\begin{equation}\label{ps100a}
	\dot{Y}(\eta_j)=0, \quad j\ge 1, \qquad \lim_{j\to\infty} \eta_j=\infty,
\end{equation}
then
\begin{equation}\label{ps100b}
	\lim_{j\to\infty} Y(\eta_j) = Y_\star := \frac{1}{m+q} \left( \frac{\beta}{\alpha} Z_\star - \sqrt{\frac{\beta^2}{\alpha^2} Z_\star^2 + 2(m+q)} \right )<0.
\end{equation}
\end{subequations}
Indeed, let $(\eta_j)_{j\ge 1}$ is a non-decreasing sequence satisfying~\eqref{ps100a}. In particular, $\dot{Y}(\eta_j) = 0$ for all $j\ge 1$, which gives, together with~\eqref{PSSyst.low},
\begin{equation} \label{ps8}
	\begin{split}
		\frac{m+q}{2} Y^2(\eta_j) & = 1 - \left( N-1+\frac{\sigma}{2} \right) X(\eta_j) Y(\eta_j)\\
		& \qquad  - X(\eta_j) Z(\eta_j) + \frac{\beta}{\alpha} Y(\eta_j) Z(\eta_j), \qquad j\ge 1.
	\end{split}
\end{equation}
Owing to~\eqref{ps3}, \eqref{ps2},  \eqref{ps8} and Young's inequality,
\begin{equation*}
	\frac{m+q}{2} Y^2(\eta_j) \le 1 + \frac{m+q}{4} Y^2(\eta_j) + \frac{[\sigma+2(N-1)]^2}{4(m+q)}  X^2(1), \qquad j\ge 1.
\end{equation*}
Therefore, $(Y(\eta_j))_{j\ge 1}$ is bounded and we may take the limit $j\to\infty$ in~\eqref{ps8} and use~\eqref{ps4} to obtain that
\begin{equation*}
	\lim_{j\to\infty} \left[ \frac{m+q}{2} Y^2(\eta_j) -  1 - \frac{\beta}{\alpha} Z_\star Y(\eta_j) \right] = 0.
\end{equation*}
Since $Y$ ranges in $(-\infty,0)$ by~\eqref{ps1}, we readily conclude that~\eqref{ps100b} holds true, and the proof of the claim is complete.

Now, either there is $\eta_\infty$ such that $\dot{Y}(\eta)\ne 0$ for $\eta\ge\eta_\infty$. Then $Y$ is monotone on $(\eta_\infty,\infty)$ and, since \eqref{ps1} excludes that it diverges to $-\infty$ as $\eta\to\infty$, there is $Y_\infty\in (-\infty,0]$ such that
\begin{equation}\label{ps5}
	\lim_{\eta\to\infty} Y(\eta) = Y_\infty, \quad \int_{\eta_\infty}^\infty |\dot{Y}(\eta)|\, d\eta = |Y_\infty - Y(\eta_\infty) | < \infty.
\end{equation}
We now infer from~\eqref{PSSyst.low}, \eqref{ps4}, and \eqref{ps5} that
\begin{equation*}
	\lim_{\eta\to\infty} \dot{Y}(\eta) = - \frac{m+q}{2} Y_\infty^2 + \frac{\beta}{\alpha} Y_\infty Z_\star + 1 = 0;
\end{equation*}
that is,
\begin{equation}\label{ps6}
	Y_\infty = \lim_{\eta\to\infty} Y(\eta) = Y_\star
\end{equation}

Or there is a non-decreasing sequence $(\eta_j)_{j\ge 1}$ satisfying~\eqref{ps100a} and, according to the claim~\eqref{ps100}, this sequence also satisfies~\eqref{ps100b}; that is,
\begin{equation}\label{ps7}
	\lim_{j\to\infty} Y(\eta_j) = Y_\star.
\end{equation}
We now define two unbounded and non-decreasing sequences $(\mu_j)_{j\ge 1}$ and $(\nu_j)_{j\ge 1}$ as follows:
\begin{equation*}
	Y(\mu_j) := \max_{[\eta_j,\eta_{j+1}]}\{Y\}, \quad Y(\nu_j) := \min_{[\eta_j,\eta_{j+1}]}\{Y\}, \qquad j\ge 1,
\end{equation*}
and set
\begin{equation*}
	\mathcal{M} := \left\{ j\ge 1\ : \ \mu_j \in (\eta_j,\eta_{j+1}) \right\}, \quad \mathcal{N} := \left\{ j\ge 1\ : \ \mu_j \in (\eta_j,\eta_{j+1}) \right\}.
\end{equation*}
Either $\mathcal{M}$ is finite and we infer from~\eqref{ps7} that
\begin{equation}\label{ps9}
	\limsup_{\eta\to\infty} Y(\eta) \le Y_\star.
\end{equation}
Or $\mathcal{M}$ is infinite and there is an infinite sequence $(j_k)_{k\ge 1}$ such that $\dot{Y}(\mu_{j_k})=0$ for all $k\ge 1$. Since $(\mu_{j_k})_{k\ge 1}$ is an increasing sequence satisfying~\eqref{ps100a}, we infer from the claim~\eqref{ps100} that $(Y(\mu_{j_k}))_{k\ge 1}$ converges to $Y_\star$ as $k\to\infty$ and conclude that~\eqref{ps9} holds true in that case as well. A similar discussion built upon the set $\mathcal{N}$ leads us to
\begin{equation}\label{ps10}
	\liminf_{\eta\to\infty} Y(\eta) \ge  Y_\star.
\end{equation}
Owing to~\eqref{ps9} and~\eqref{ps10}, we have established that
\begin{equation*}
	\lim_{\eta\to\infty} Y(\eta) = Y_\star,
\end{equation*}
which gives, together with~\eqref{change.small} and~\eqref{change.small1},
$$
\lim\limits_{\xi\to\xi_0} \left(f^{(m-q)/2}\right)'(\xi) =\frac{m-q}{2\sqrt{m}} Y_\star \xi_0^{\sigma/2}.
$$
Hence, after integration over $(\xi,\xi_0)$,
$$
\lim\limits_{\xi\to\xi_0} f^{(m-q)/2}(\xi) = - \frac{m-q}{2\sqrt{m}} Y_\star \xi_0^{\sigma/2} (\xi_0-\xi),
$$
and we complete the proof by expressing $Y_\star$ in terms of $Z_\star$ according to the range of $m+q$, see~\eqref{ps4} and~\eqref{ps100b}.
\end{proof}

\medskip

We next turn to the complementary case $m+q>2$ and again use a phase space technique, but with a different dynamical system adapted to this range of exponents.

\begin{proposition}\label{prop.interf.high}
Let $a\in\mathcal{B}$ and assume that $m+q\geq2$. Then (with $\xi_0=\xi_0(a)$)
$$
f(\xi;a) = K_{3}\xi_0^{(\sigma-1)/(1-q)} (\xi_0-\xi)^{1/(1-q)} + o\Big((\xi_0-\xi)^{1/(1-q)}\Big), \qquad {\rm as} \ \xi\to\xi_0,
$$
where
$$
K_3 := \left( \frac{1-q}{\beta} \right)^{1/(1-q)}.
$$
\end{proposition}

\begin{proof}
We introduce a change of unknown functions following \cite{IS20, IMS22} by setting
\begin{equation}\label{change.high}
\begin{split}
\mathcal{X}(\xi) & := \frac{m}{\alpha} \xi^{-2} f^{m-1}(\xi), \\
\mathcal{Y}(\xi) & := \frac{m}{\alpha} \xi^{-1} f^{m-2}(\xi)f'(\xi), \\
\mathcal{Z}(\xi) & := \frac{m}{\alpha^2} \xi^{\sigma-2} f^{m+q-2}(\xi),
\end{split}
\end{equation}
together with a new independent variable
\begin{equation}\label{PSvar.high}
\eta(\xi) := \frac{\alpha}{m} \int_0^\xi \frac{\xi_*}{f^{m-1}(\xi_*)}\,d\xi_*, \qquad \xi\in [0,\xi_0).
\end{equation}
On the one hand, arguing as in the proof of Proposition~\ref{prop.interf.low}, we get
\begin{equation}\label{interm16}
	\begin{split}
		f'(\xi) & = \frac{\alpha}{m} \xi f^{2-m}(\xi) \mathcal{Y}(\xi),\\
		f''(\xi) & = \frac{\alpha}{m} f^{2-m}(\xi) \mathcal{Y}(\xi) + \frac{\alpha}{m} \xi f^{2-m}(\xi) \mathcal{Y}(\xi) + \frac{\alpha(2-m)}{m} \xi f^{1-m}(\xi) f'(\xi) \mathcal{Y}(\xi)\\
		&=\frac{\alpha}{m} f^{2-m}(\xi) \mathcal{Y}(\xi) + \frac{\alpha}{m} \xi f^{2-m}(\xi) \mathcal{Y}(\xi) + \frac{\alpha^2(2-m)}{m^2} \xi^2 f^{3-2m}(\xi) \mathcal{Y}^2(\xi).
	\end{split}
\end{equation}
Replacing the derivatives obtained in \eqref{interm16} into Eq.~\eqref{ODE2} and changing at the end to
\begin{equation}\label{change.high1}
	(\mathcal{X},\mathcal{Y},\mathcal{Z}) := (X\circ\eta, Y\circ\eta, Z\circ\eta)
\end{equation}
leads, after technical but straightforward calculations, to the following autonomous dynamical system
\begin{equation}\label{PSSyst.high}
	\left\{\begin{split}
		\dot{X}&=X[(m-1)Y-2X]\\
		\dot{Y}&=-Y^2+\frac{\beta}{\alpha}Y-X-NXY+Z\\ \dot{Z}&=Z[(m+q-2)Y+(\sigma-2)X],
	\end{split}\right.
\end{equation}
where the derivatives marked by the dot in the left hand side are taken with respect to the variable $\eta$ introduced in~\eqref{PSvar.high}.

On the other hand, since $a\in\mathcal{B}$, the function $f$ maps a left neighborhood $(\xi_0-\delta,\xi_0)$ of $\xi_0$ on $(0,1)$ and it follows from~\eqref{i3} and the property $m-1>1-q$ that
\begin{equation*}
	\frac{\xi_*}{f^{m-1}(\xi_*)} \ge \frac{\xi_*}{f^{1-q}(\xi_*)} \ge C_2^{q-1} \frac{\xi_*}{\xi_0-\xi_*}, \qquad \xi_*\in (\xi_0-\delta,\xi_0),
\end{equation*}
so that $\eta$ is a diffeomorphism from $[0,\xi_0)$ onto $[0,\infty)$.  We shall then study the behavior of $(X,Y,Z)(\eta)$ as $\eta\to\infty$.

As a preliminary observation, we note that the property $a\in\mathcal{B}$, \eqref{change.high} and Lemma~\ref{lem.i3} entail that
\begin{equation}\label{z2a}
	X \ge 0, \quad Y \le 0, \quad Z\ge 0,
\end{equation}
and
\begin{equation}\label{z2b}
	\lim_{\eta\to\infty} X(\eta) = 0, \quad \limsup_{\eta\to\infty} Y(\eta) = 0, \quad  \lim_{\eta\to\infty} Z(\eta) = 0.
\end{equation}
Owing to~\eqref{z2a} and~\eqref{z2b}, we next argue as in the proof of Proposition~\ref{prop.interf.low} to identify the behavior of $Y(\eta)$ as $\eta\to\infty$ and conclude that
\begin{equation}\label{z2c}
	\lim_{\eta\to\infty} Y(\eta) = 0.
\end{equation}
Recalling~\eqref{change.high} and~\eqref{change.high1}, we readily infer from~\eqref{z2c} that
\begin{equation}\label{z6}
	\lim_{\xi\to\xi_0} \left( f^{m-1} \right)'(\xi) = 0.
\end{equation}
Now, recalling the definition~\eqref{interm11} of $H$, we deduce from~\eqref{z6} that
\begin{equation}\label{z7}
	- H(\xi) = \beta \xi^N f(\xi) \left( 1 - \frac{m}{m-1} \frac{\left( f^{m-1}\right)'(\xi)}{\xi} \right) \sim \beta \xi_0^N f(\xi), \ \ {\rm as} \ \xi\to\xi_0.
\end{equation}
Also, since $a\in\mathcal{B}$ and $q\in (0,1)$, it follows from the formula~\eqref{interm9} for $H'$ and~\eqref{z6} that
\begin{equation*}
	H'(\xi) = \xi^{\sigma+N-1} f^q(\xi) \left( 1 - (\alpha+N\beta) \xi^{-\sigma} f^{1-q}(\xi) \right) \sim \xi_0^{\sigma+N-1} f^q(\xi), \ \ {\rm as} \ \xi\to\xi_0.
\end{equation*}
Consequently,
\begin{equation*}
	\lim\limits_{\xi\to\xi_0} \frac{H'(\xi)}{(-H)^q(\xi)} = \frac{1}{\beta^q} \xi_0^{\sigma+N(1-q)-1};
\end{equation*}
that is,
\begin{equation*}
	\lim\limits_{\xi\to\xi_0} \left[ (- H)^{1-q}\right]'(\xi) = - \frac{(1-q)}{\beta^q} \xi_0^{\sigma+N(1-q)-1}.
\end{equation*}
Since $H(\xi_0)=0$, we obtain, after integration over $(\xi,\xi_0)$,
\begin{equation*}
	(- H)^{1-q}(\xi) \sim \frac{(1-q)}{\beta^q} \xi_0^{\sigma+N(1-q)-1} (\xi_0-\xi), \ \ {\rm as} \ \xi\to\xi_0.
\end{equation*}
Owing to~\eqref{z7}, we end up with
\begin{equation*}
	\beta \xi_0^N f(\xi) \sim \beta K_3 \xi_0^{(\sigma+N(1-q)-1)/(1-q)} (\xi_0-\xi)^{1/(1-q)}, \ \ {\rm as} \ \xi\to\xi_0,
\end{equation*}
and the proof of Proposition~\ref{prop.interf.high} is complete.
\end{proof}

With all these preparations, we are now in a position to prove the uniqueness of the self-similar solution and thus complete the proof of Theorem~\ref{th.uniqSS}.

\subsection{Uniqueness of the self-similar profile}\label{subsec.uniq}

We begin with a general monotonicity property of the profiles $f(\cdot;a)$ with respect to the parameter $a>0$. Recalling the definitions of $\xi_0(a)$ and $\xi_1(a)$ introduced in~\eqref{support} and~\eqref{support.der}, respectively, we have the following ordering property.

\begin{lemma}\label{lem.monot}
Let $0<a_1<a_2<\infty$. Then $f(\xi;a_1)<f(\xi;a_2)$ for any $\xi\in [0,\xi_1(a_1))$. In other words, two solutions $f(\cdot;a_1)$ and $f(\cdot;a_2)$ to the Cauchy problem~\eqref{ODE2}-\eqref{init.cond.ODE2} remain ordered at least as long as the one with smaller initial value is decreasing and positive.
\end{lemma}

\begin{proof}
Set $F_i := F(\cdot;a_i) = (f^m)(\cdot;a_i)$ for $i=1,2$ and consider $X\in (0,\xi_1(a_1))$. Since $a_1<a_2$, it follows that $F_1<F_2$ in a right neighborhood of $\xi=0$ and we define
\begin{equation*}
\xi_{*}:=\inf\{\xi\in(0,X): F_1(\xi)=F_2(\xi)\}>0.
\end{equation*}
Thus $F_1(\xi)<F_2(\xi)$ for $\xi\in[0,\xi_{*})$. Assume for contradiction that $\xi_{*}<X$, so that $F_1(\xi_{*})=F_2(\xi_{*})$. We argue as in \cite{IMS22b, YeYin} and define, for $\lambda\in[0,1]$, the rescaling
\begin{equation*}
G_{\lambda}(\xi)=\lambda^{-2m/(m-1)}F_1(\lambda\xi), \qquad \xi\in[0,X].
\end{equation*}
By straightforward calculations, $G_{\lambda}$ is a solution to the ordinary differential equation
\begin{equation}\label{ODE.resc.monot}
\begin{split}
G_{\lambda}''(\xi) + \frac{N-1}{\xi} G_{\lambda}'(\xi) & + \alpha (G_{\lambda}^{1/m})(\xi) - \beta \xi (G_{\lambda}^{1/m})'(\xi)\\
& - \xi^{\sigma} \lambda^{(\sigma(m-1)+2(q-1))/(m-1)}G_{\lambda}^{q/m}(\xi)=0,
\end{split}
\end{equation}
with initial condition
\begin{equation}\label{init.cond.ODE.monot}
G_{\lambda}(0)=\lambda^{-2m/(m-1)}a_1, \qquad G_{\lambda}'(0)=0.
\end{equation}
Let us first remark that, if $0<\lambda<\lambda' \le 1$, then the monotonicity of $F_1$ on $[0,X]$ guarantees that
$$
F_1(\lambda'\xi)<F_1(\lambda\xi), \qquad \xi\in[0,X],
$$
which implies
\begin{equation*}
\begin{split}
G_{\lambda}(\xi)&=\lambda^{-2m/(m-1)}F_1(\lambda\xi)>(\lambda')^{-2m/(m-1)}F_1(\lambda\xi)\\
&>(\lambda')^{-2m/(m-1)}F_1(\lambda'\xi)=G_{\lambda'}(\xi),
\end{split}
\end{equation*}
whence
\begin{equation*}
F_1\leq G_{\lambda'}<G_{\lambda}, \quad {\rm in} \ [0,X], \qquad 0<\lambda<\lambda'\leq 1,
\end{equation*}
and
\begin{align}
\lim\limits_{\lambda\to 0} \min\limits_{[0,X]} G_{\lambda} = \lim\limits_{\lambda\to 0} G_{\lambda}(X) & = \lim\limits_{\lambda\to 0} \lambda^{-2m/(m-1)} F_1(\lambda X) \nonumber \\
& \ge F_1\left( \frac{X}{2} \right) \lim\limits_{\lambda\to 0} \lambda^{-2m/(m-1)} = \infty. \label{interm22}
\end{align}
We can thus define
\begin{equation}\label{interm23}
\lambda_0:=\sup\{\lambda\in(0,1): F_2(\xi)<G_{\lambda}(\xi), \xi\in[0,\xi_{*}]\}
\end{equation}
and readily infer from \eqref{interm22} and the fact that $F_1(\xi)<F_2(\xi)$ for $\xi\in(0,\xi_{*})$ that $\lambda_0\in(0,1)$. This gives in particular that $F_2(\xi)\leq G_{\lambda_0}(\xi)$ for any $\xi\in[0,\xi_{*}]$ and that there exists some $\eta\in[0,\xi_*]$ such that $F_2(\eta)=G_{\lambda_0}(\eta)$ (otherwise, $G_{\lambda_0}-F_2>0$ on the compact set $[0,\xi_{*}]$, which contradicts the optimality of $\lambda_0$ in~\eqref{interm23}). Assume first that $\eta=\xi_{*}$. Then we readily get a contradiction from
$$
G_{\lambda_0}(\eta) =  G_{\lambda_0}(\xi_*) = \lambda_0^{-2m/(m-1)} F_1(\lambda_0\xi_*) > F_1(\xi_*) = F_2(\xi_*) = F_2(\eta).
$$
Assume next that $\eta\in(0,\xi_{*})$. Then the function $G_{\lambda_0}-F_2$ on $[0,\xi_{*}]$ has a minimum point at $\eta\in (0,\xi_{*})$ and thus
\begin{equation}\label{interm24}
G_{\lambda_0}(\eta)=F_2(\eta), \ \ G_{\lambda_0}'(\eta)=F_2'(\eta), \ \ G_{\lambda_0}''(\eta)\geq F_2''(\eta).
\end{equation}
We then deduce from the equations~\eqref{ODE2} and~\eqref{ODE.resc.monot} solved by $F_2$ and $G_{\lambda_0}$, respectively, and from~\eqref{interm24} that
\begin{equation*}
\begin{split}
0&=F_2''(\eta)+\frac{N-1}{\eta}F_2'(\eta)+\alpha F_2^{1/m}(\eta)-\beta\eta(F_2^{1/m})'(\eta)-\eta^{\sigma}F_2^{q/m}(\eta)\\
&\leq G_{\lambda_0}''(\eta)+\frac{N-1}{\eta}G_{\lambda_0}'(\eta)+\alpha G_{\lambda_0}^{1/m}(\eta)-\beta\eta(G_{\lambda_0}^{1/m})'(\eta)-\eta^{\sigma}G_{\lambda_0}^{q/m}(\eta)\\
&=\eta^{\sigma}\left[\lambda_0^{\sigma+2(q-1)/(m-1)}-1\right]G_{\lambda_0}^{q/m}(\eta)<0,
\end{split}
\end{equation*}
since $\eta\neq0$, $G_{\lambda_0}(\eta)>0$, $\lambda_0\in(0,1)$ and $\sigma(m-1)+2(q-1)>0$ in our range of exponents~\eqref{range.exp}. We thus reach again a contradiction.

We are thus left with $\eta=0$, which gives
\begin{equation}\label{interm25}
G_{\lambda_0}(0)=F_2(0), \qquad G_{\lambda_0}(\xi)>F_2(\xi), \ \ \xi\in(0,\xi_{*}].
\end{equation}
It readily follows from the equality at $\xi=0$ in \eqref{interm25} that
\begin{equation}\label{interm26}
a_2^m=a_1^m\lambda_0^{-2m/(m-1)} \qquad {\rm or \ equivalently} \qquad a_2=a_1\lambda_0^{-2/(m-1)}.
\end{equation}
We then argue as in \cite[Section~5]{IMS22b} by using the expansions for $\xi$ small obtained in Section~\ref{subsec.basic} to reach a contradiction. The proof is now split into two cases.

\medskip

\noindent \textbf{Case 1. $\sigma\not\in\mathbb{N}$}. We recall that, as $\xi\to0$, \eqref{interm8} and \eqref{exp.smallF1} imply that
$$
F_i(\xi)=\sum\limits_{j=0}^{k_0+2}a_i^{m-j(m-1)/2}\Omega_j\xi^{j}+\frac{a_i^q}{(\sigma+N)(\sigma+2)}\xi^{\sigma+2}+o(\xi^{\sigma+2}),
$$
for $i=1,2$ and where $k_0$ is the integer part of $\sigma$. We then infer from \eqref{interm26} and the previous expansion applied to $F_1$ that
\begin{equation*}
\begin{split}
G_{\lambda_0}(\xi)&=\lambda_0^{-2m/(m-1)}\left[\sum\limits_{j=0}^{k_0+2}a_1^{m-j(m-1)/2}\lambda_0^j\Omega_j\xi^{j}+\frac{a_1^q\lambda_0^{\sigma+2}}{(\sigma+N)(\sigma+2)}\xi^{\sigma+2}+o(\xi^{\sigma+2})\right]\\
&=\sum\limits_{j=0}^{k_0+2}a_1^{m}\lambda_0^{-2m/(m-1)}\left(a_1\lambda_0^{-2/(m-1)}\right)^{-j(m-1)/2}\Omega_j\xi^{j}\\
&\qquad + \frac{a_1^q\lambda_0^{[\sigma(m-1)-2]}/(m-1)}{(\sigma+N)(\sigma+2)} \xi^{\sigma+2} + o(\xi^{\sigma+2})\\
&=\sum\limits_{j=0}^{k_0+2}a_2^{m-j(m-1)/2}\Omega_j\xi^{j}+\frac{a_2^q\lambda_0^{[\sigma(m-1)+2(q-1)]/(m-1)}}{(\sigma+N)(\sigma+2)}\xi^{\sigma+2}+o(\xi^{\sigma+2})\\
&=F_2(\xi)-\frac{a_2^q}{(\sigma+N)(\sigma+2)}\xi^{\sigma+2}+\frac{a_2^q\lambda_0^{[\sigma(m-1)+2(q-1)]/(m-1)}}{(\sigma+N)(\sigma+2)}\xi^{\sigma+2}+o(\xi^{\sigma+2})\\
&=F_2(\xi)+\frac{a_2^q}{(\sigma+N)(\sigma+2)}\xi^{\sigma+2}\left[\lambda_0^{[\sigma(m-1)+2(q-1)]/(m-1)}-1\right]+o(\xi^{\sigma+2}).
\end{split}
\end{equation*}
Since $\lambda_0\in(0,1)$ and $\sigma(m-1)+2(q-1)>0$ by~\eqref{range.exp}, we conclude that $G_{\lambda_0}(\xi)<F_2(\xi)$ in a right neighborhood of $\xi=0$, which is a contradiction to \eqref{interm25}.

\medskip

\noindent \textbf{Case 2. $\sigma\in\mathbb{N}$}. We recall that, as $\xi\to0$, \eqref{interm8} and \eqref{exp.smallF2} imply that
$$
F_i(\xi)=\sum\limits_{j=0}^{k_0+3}a_i^{m-j(m-1)/2}\Omega_j\xi^{j}+\frac{a_i^q}{(\sigma+N)(\sigma+2)}\xi^{\sigma+2}+o(\xi^{\sigma+2}),
$$
for $i=1,2$ and where $k_0=\sigma-1$, and we proceed as in Case~1 to obtain a similar contradiction to \eqref{interm25}.

We have thus reached a contradiction with our initial assumption $\xi_{*}<X$. It thus follows that $\xi_{*}=X$ and $F_1(\xi)<F_2(\xi)$ for $\xi\in[0,X)$. The proof is ended by noticing that $X$ has been chosen arbitrarily in $(0,\xi_1(a_1))$.
\end{proof}

We are now in a position to complete the proof of the uniqueness part of Theorem~\ref{th.uniqSS} by specializing our analysis for elements $a\in\mathcal{B}$.

\begin{proof}[Proof of Theorem~\ref{th.uniqSS}: Uniqueness]
Assume for contradiction that there are $a_1\in\mathcal{B}$ and $a_2\in\mathcal{B}$ such that $a_1<a_2$. Since we know from Lemma~\ref{lem.B} that $\xi_1(a)=\xi_0(a)$ and $f'(\cdot;a)<0$ on $(0,\xi_0(a))$ for $a\in\mathcal{B}$, we infer from Lemma~\ref{lem.monot} that $F_1(\xi)<F_2(\xi)$ for any $\xi\in[0,\xi_0(a_1))$. In particular, $\xi_0(a_2)\ge \xi_0(a_1)$ and, since $F_1(\xi)=0<F_2(\xi)$ for $\xi\in [\xi_0(a_1),\xi_0(a_2))$, we conclude that
\begin{equation*}
	F_1(\xi) < F_2(\xi), \qquad \xi\in [0,\xi_0(a_2)).
\end{equation*}	
As in the proof of Lemma~\ref{lem.monot}, we set
\begin{equation*}
	G_\lambda(\xi) := \lambda^{-2m/(m-1)} F_1(\lambda \xi), \qquad (\xi,\lambda)\in [0,\xi_0(a_2)]\times (0,1],
\end{equation*}	
and define
\begin{equation}\label{up2}
	\lambda_0 := \sup\left\{ \lambda\in (0,1]\ :\ F_2(\xi)<G_\lambda(\xi),\ \xi\in [0,\xi_0(a_2)) \right\} \in (0,1),
\end{equation}
the existence of $\lambda_0$ being guaranteed by the property
\begin{align*}
	\lim\limits_{\lambda\to 0} \min_{[0,\xi_0(a_2)]} G_\lambda & = \lim\limits_{\lambda\to 0} G_\lambda(\xi_0(a_2)) = \lim\limits_{\lambda\to 0} \lambda^{-2m/(m-1)} F_1(\lambda \xi_0(a_2)) \\
	& \ge F_1\left( \frac{\xi_0(a_1)}{2} \right) \lim\limits_{\lambda\to 0} \lambda^{-2m/(m-1)} = \infty.
\end{align*}
According to the definition~\eqref{up2} of $\lambda_0$, there is $\eta\in [0,\xi_0(a_2)]$ such that $F_2(\eta)=G_{\lambda_0}(\eta)$ and $F_2\le G_{\lambda_0}$ on $[0,\xi_0(a_2)]$. Arguing as in the proof of Lemma~\ref{lem.monot} discards that $\eta$ lies in $[0,\xi_0(a_2))$ and we end up with $\eta=\xi_0(a_2)$; that is,	
\begin{equation}\label{interm27}
F_2(\xi_0(a_2))=G_{\lambda_0}(\xi_0(a_2))=0, \qquad 0<F_2(\xi)<G_{\lambda_0}(\xi), \ \ \xi\in[0,\xi_0(a_2)).
\end{equation}
First of all, we readily infer from the equality of the supports in \eqref{interm27} that
\begin{equation} \label{up3}
	\xi_0(a_1)=\lambda_0\xi_0(a_2).
\end{equation}
We now split the end of the analysis into three cases according to the range of $m+q-2$, taking into account that the local behavior at the interface is different in each case.

\medskip

\noindent \textbf{Case~1. $m+q<2$}. We recall that in this case Proposition~\ref{prop.interf.low} gives
$$
F_i(\xi) = K_1^m \xi_0^{m\sigma/(m-q)}(a_i) (\xi_0(a_i)-\xi)^{2m/(m-q)} + o((\xi_0(a_i)-\xi)^{2m/(m-q)})
$$
as $\xi\to\xi_0(a_i)$, $i=1,2$. Using~\eqref{up3}, we compute the expansion of $G_{\lambda_0}(\xi)$ as $\xi\to\xi_0(a_2)$ and find
\begin{align*}
G_{\lambda_0}(\xi) & = \lambda_0^{-2m/(m-1)} K_1^m \xi_0(a_1)^{m\sigma/(m-q)} (\xi_0(a_1)-\lambda_0\xi)^{2m/(m-q)}\\
& \qquad + o((\xi_0(a_1)-\lambda_0\xi)^{2m/(m-q)})\\
&=\lambda_0^{-2m/(m-1)}K_1^m (\lambda_0\xi_0(a_2))^{m\sigma/(m-q)} \lambda_0^{2m/(m-q)} (\xi_0(a_2)-\xi)^{2m/(m-q)}\\
& \qquad +o((\xi_0(a_2)-\xi)^{2m/(m-q)})\\
&=\lambda_0^{m[\sigma(m-1)+2(q-1)]/(m-1)(m-q)}K_1^m (\xi_0(a_2)-\xi)^{2m/(m-q)}\\
&\qquad +o((\xi_0(a_2)-\xi)^{2m/(m-q)}) \\
&= \lambda_0^{m[\sigma(m-1)+2(q-1)]/(m-1)(m-q)} F_2(\xi) + o((\xi_0(a_2)-\xi)^{2m/(m-q)}).
\end{align*}
Since $\lambda_0\in(0,1)$ and $\sigma(m-1)+2(q-1)>0$ in our range of exponents \eqref{range.exp}, we easily infer from the above formula that $G_{\lambda_0}(\xi)<F_2(\xi)$ in a left neighborhood of their common edge of the support $\xi_0(a_2)$, in contradiction with~\eqref{interm27}.

\medskip

\noindent \textbf{Case~2. $m+q=2$}. We first recall that, in that peculiar case,  Proposition~\ref{prop.interf.low} gives
\begin{align*}
	F_i(\xi) & = K_1^m \xi_0^{m\sigma/(m-q)}(a_i) K_2\left( \xi_0^{(2-\sigma)/2}(a_i) \right)^m (\xi_0(a_i)-\xi)^{2m/(m-q)} \\
	& \qquad + o\Big((\xi_0(a_i)-\xi)^{2m/(m-q)}\Big)
\end{align*}
as $\xi\to\xi_0(a_i)$, $i=1,2$. Using \eqref{up3}, we compute the expansion of $G_{\lambda_0}(\xi)$ as $\xi\to\xi_0(a_2)$ and obtain
\begin{align*}
	G_{\lambda_0}(\xi) & = \lambda_0^{-2m/(m-1)} K_1^m \xi_0^{m\sigma/(m-q)}(a_1) \\
	& \hspace{2cm} \times K_2^m\left( \xi_0^{(2-\sigma)/2}(a_1) \right) (\xi_0(a_1)-\lambda_0\xi)^{2m/(m-q)} \\
	& \qquad + o\Big((\xi_0(a_1)-\lambda_0\xi)^{2m/(m-q)}\Big) \\
	& = \lambda_0^{-2m/(m-1)} K_1^m (\lambda_0\xi_0(a_2))^{m\sigma/(m-q)} \\
	& \hspace{2cm} \times K_2^m\left( (\lambda_0\xi_0(a_2))^{(2-\sigma)/2} \right) \lambda_0^{2m/(m-q)}(\xi_0(a_2)-\xi)^{2m/(m-q)} \\
	& \qquad + o\Big((\xi_0(a_2)- \xi)^{2m/(m-q)}\Big) \\
	& = \lambda_0^{m[\sigma(m-1)+2(q-1)]/(m-1)(m-q)} K_1^m \xi_0^{m\sigma/(m-q)}(a_2) \\
	& \hspace{2cm} \times K_2^m\left( (\lambda_0\xi_0(a_2))^{(2-\sigma)/2} \right) (\xi_0(a_2)-\xi)^{2m/(m-q)} \\
	& \qquad + o\Big((\xi_0(a_2)- \xi)^{2m/(m-q)}\Big) \\
	& = \lambda_0^{m[\sigma(m-1)+2(q-1)]/(m-1)(m-q)} \left[ \frac{ K_2^m\left( (\lambda_0\xi_0(a_2))^{(2-\sigma)/2} \right)}{K_2^m\left( \xi_0(a_2)^{(2-\sigma)/2} \right)} \right] F_2(\xi) \\
	& \qquad + o\Big((\xi_0(a_2)- \xi)^{2m/(m-q)}\Big) .
\end{align*}

Introducing
\begin{equation*}
	\kappa(z) := \frac{1+\sqrt{1+\kappa_0}}{1+\sqrt{1 + \kappa_0 z^2}} z, \quad z\ge 0, \qquad \kappa_0 := \frac{2m(m+q) \xi_0^{2-\sigma}(a_2)}{\beta^2},
\end{equation*}
the expansion of $G_{\lambda_0}(\xi)$ as $\xi\to\xi_0(a_2)$ reads
\begin{equation}\label{up4}
	\begin{split}
	G_{\lambda_0}(\xi) & = \lambda_0^{m[\sigma(m-1)+2(q-1)]/(m-1)(m-q)}\kappa^{2m/(m-q)}\left( \lambda_0^{(\sigma-2)/2} \right) F_2(\xi) \\
	& \qquad + o\Big((\xi_0(a_2)- \xi)^{2m/(m-q)}\Big) .
	\end{split}
\end{equation}
Now, $\kappa(0)=0$, $\kappa(1) = 1$, and
\begin{equation*}
	\kappa'(z) = \frac{(1+\sqrt{1+\kappa_0})(1+\sqrt{1+ \kappa_0 z^2})}{\sqrt{1+ \kappa_0 z^2} (1+\sqrt{1+ \kappa_0 z^2})^2}>0.
\end{equation*}
Consequently, since $\sigma>2$ by~\eqref{range.exp} and the property $m+q=2$, we have $\lambda_0^{(\sigma-2)/2}<1$ and thus $\kappa\left( \lambda_0^{(\sigma-2)/2} \right) <1$ as well. Using once more~\eqref{range.exp}, we end up with
\begin{equation*}
	\lambda_0^{m[\sigma(m-1)+2(q-1)]/(m-1)(m-q)}\kappa^{2m/(m-q)}\left( \lambda_0^{(\sigma-2)/2} \right) < 1,
\end{equation*}
which, together with~\eqref{up4}, contradicts~\eqref{interm27}.

\medskip

\noindent \textbf{Case~3. $m+q>2$}. In this case, by Proposition~\ref{prop.interf.high},
$$
F_i(\xi)=K_3^m(m,q,\sigma) \xi_0^{m(\sigma-1)/(1-q)}(a_i) (\xi_0(a_i)-\xi)^{m/(1-q)} + o((\xi_0(a_i)-\xi)^{m/(1-q)})
$$
as $\xi\to\xi_0(a_i)$, $i=1,2$. Using \eqref{up3}, we compute the expansion of $G_{\lambda_0}(\xi)$ as $\xi\to\xi_0(a_2)$ and find
\begin{align*}
G_{\lambda_0}(\xi)&=\lambda_0^{-2m/(m-1)} K_3^m(m,q,\sigma) \xi_0^{m(\sigma-1)/(1-q)}(a_1) (\xi_0(a_1)-\lambda_0\xi)^{m/(1-q)}\\
& \qquad +o((\xi_0(a_1)-\lambda_0\xi)^{m/(1-q)})\\
&=\lambda_0^{-2m/(m-1)}K_3^m(m,q,\sigma) (\lambda_0\xi_0(a_2))^{m(\sigma-1)/(1-q)} \lambda_0^{m/(1-q)}(\xi_0(a_2)-\xi)^{m/(1-q)}\\
&\qquad +o((\xi_0(a_2)-\xi)^{m/(1-q)})\\
&=\lambda_0^{m[\sigma(m-1)+2(q-1)]/(m-1)(1-q)}K_3^m(m,q,\sigma) \xi_0^{m(\sigma-1)/(1-q)}(a_2) (\xi_0(a_2)-\xi)^{m/(1-q)}\\
& \qquad + o((\xi_0(a_2)-\xi)^{m/(1-q)}) \\
& = \lambda_0^{m[\sigma(m-1)+2(q-1)]/(m-1)(1-q)} F_2(\xi) + o((\xi_0(a_2)-\xi)^{m/(1-q)}).
\end{align*}
Since $\lambda_0\in(0,1)$ and $\sigma(m-1)+2(q-1)>0$ in our range of exponents \eqref{range.exp}, we again deduce that $G_{\lambda_0}(\xi)<F_2(\xi)$ in a left neighborhood of $\xi_0(a_2)$, in contradiction with \eqref{interm27}.

\medskip

These contradictions show that the set $\mathcal{B}$ can have at most one element $a>0$ and the proof is complete.
\end{proof}

\section{Large time behavior of solutions}\label{sec.asympt}

This section is devoted to the behavior as $t\to\infty$ of weak solutions to Eq.~\eqref{eq1}, completing the proof of Theorem~\ref{th.asympt}, which is now rather short. In order to fix the notation, let $f^{*}=f(\cdot;a^*)$ be the unique self-similar profile according to Theorem ~\ref{th.uniqSS}, with $a^*=f^{*}(0)$ and let $\xi_0^{*}=\xi_0(a^*)$ be the edge of the support of $f^{*}$.

\begin{proof}[Proof of Theorem~\ref{th.asympt}]
Let $u_0\in L_+^{\infty}(\real^N)$ with the property that there exists $\delta>0$ and $r>0$ such that $u_0(x)\geq\delta$ for $x\in B(0,r)$. We first pick $\tau_{\infty}>1$ such that
\begin{equation*}
r\tau_{\infty}^{\beta}>\xi_0^{*}=\xi_0(a^*), \qquad \tau_{\infty}^{\alpha}\delta>a^*.
\end{equation*}
Then, for $x\in B(0,r)$, we have
$$
\tau_{\infty}^{-\alpha}f^{*}(|x|\tau_{\infty}^{\beta})<\tau_{\infty}^{-\alpha}a^*<\delta\leq u_0(x),
$$
while, for $x\in\real^N\setminus B(0,r)$, one has $|x|\tau_{\infty}^{\beta} > \xi_0^*$ and thus
$$
\tau_{\infty}^{-\alpha}f^{*}(|x|\tau_{\infty}^{\beta})=0\leq u_0(x).
$$
The comparison principle then entails that
\begin{equation}\label{lower.bound}
u(t,x)\geq(\tau_{\infty}+t)^{-\alpha}f^{*}(|x|(\tau_{\infty}+t)^{\beta}), \qquad (t,x)\in[0,\infty)\times\real^N.
\end{equation}
Next, according to Theorem~\ref{th.wp}, $u(1)$ is compactly supported in $B(0,R)$ for some $R>0$. We then pick $\tau_0\in (0,1)$ such that
\begin{equation*}
\tau_0^{-\alpha}f^{*}(a^*/2)\geq\|u_0\|_{\infty}, \qquad R\tau_0^{\beta}\leq\frac{a^*}{2}.
\end{equation*}
Then, for $x\in B(0,R)$, the monotonicity of $f^{*}$ and Theorem~\ref{th.wp} imply that
$$
u(1,x) \le \|u(1)\|_\infty \leq\|u_0\|_{\infty}\leq\tau_0^{-\alpha} f^{*}\left(\frac{a^*}{2}\right)\leq\tau_0^{-\alpha}f^{*}(R\tau_0^{\beta})\leq\tau_0^{-\alpha}f^{*}(|x|\tau_0^{\beta}),
$$
while, for $x\in\real^N\setminus B(0,R)$, we have
$$
\tau_0^{-\alpha}f^{*}(|x|\tau_0^{\beta})\geq 0=u(1,x).
$$
Applying again the comparison principle gives
\begin{equation}\label{upper.bound}
u(t+1,x)\leq(\tau_0+t)^{-\alpha}f^{*}(|x|(\tau_0+t)^{\beta}), \qquad (t,x)\in[0,\infty)\times\real^N.
\end{equation}
We next introduce the self-similar variables
\begin{equation}\label{SS.var}
u(t,x)=t^{-\alpha}v(s,y), \qquad s=\ln{t}, \qquad y=x t^{\beta}
\end{equation}
for $(t,x)\in[1,\infty)\times\real^N$. We infer from~\eqref{lower.bound} and~\eqref{upper.bound} that
$$
\left(\frac{t}{\tau_{\infty}+t}\right)^{\alpha}f^{*}\left(|y|\left(\frac{\tau_{\infty}+t}{t}\right)^{\beta}\right)\leq v(s,y)\leq \left(\frac{t}{\tau_{0}-1+t}\right)^{\alpha}f^{*}\left(|y|\left(\frac{\tau_{0}-1+t}{t}\right)^{\beta}\right),
$$
or equivalently in the new variables~\eqref{SS.var}
\begin{align*}
\left(\frac{1}{1+\tau_{\infty}e^{-s}}\right)^{\alpha}& f^{*}\left( |y|(1+\tau_{\infty}e^{-s})^{\beta} \right)\leq v(s,y)\\
&\leq \left(\frac{1}{1+(\tau_{0}-1)e^{-s}}\right)^{\alpha} f^{*}\left( |y|(1+(\tau_{0}-1)e^{-s})^{\beta} \right).
\end{align*}
It then follows from the uniform continuity of $f^{*}$ that
$$
\lim\limits_{s\to\infty}\|v(s)-f^{*}\|_{\infty}=0,
$$
and the proof ends by undoing the self-similar change of variable~\eqref{SS.var} to obtain~\eqref{asympt.beh}.
\end{proof}

As announced in the Introduction, we close the paper with a short result showing that the positivity in a neighborhood of $x=0$ is essential for Theorem~\ref{th.asympt} to hold true.

\begin{proposition}\label{prop.stat}
The function
\begin{equation}\label{statsol}
U(x)=A|x|^{(\sigma+2)/(m-q)}, \qquad A=\left[\frac{(m-q)^2}{m(\sigma+2)[m(\sigma+N)-q(N-2)]}\right]^{1/(m-q)}
\end{equation}
is an explicit but unbounded stationary solution to Eq.~\eqref{eq1} with $U(0)=0$. Moreover, if $u$ is the solution to the Cauchy problem~\eqref{eq1}, \eqref{init.cond} and if $u_0(x)\leq U(x)$ for any $x\in\real^N$, then
$$
u(t,0)=0 \qquad {\rm for \ any} \ t>0.
$$
\end{proposition}

\begin{proof}
Let $x\in\real^N$. A direct calculation gives on the one hand that
\begin{equation}\label{interm28}
\Delta U^m(x) = A^m \frac{m(\sigma+2)[m(\sigma+N)-q(N-2)]}{(m-q)^2} |x|^{[m(\sigma+2)-2(m-q)]/(m-q)},
\end{equation}
and on the other hand that
\begin{equation}\label{interm29}
|x|^{\sigma}U^q(x) = A^q |x|^{[q(\sigma+2)+\sigma(m-q)]/(m-q)}.
\end{equation}
Since
$$
m(\sigma+2) - 2(m-q) = q(\sigma+2) + \sigma(m-q) = \sigma m + 2 q
$$
and taking into account the expression of $A$ in~\eqref{statsol}, we readily infer from~\eqref{interm28} and~\eqref{interm29} that
\begin{equation}\label{interm30}
\partial_t U -\Delta U^m + |x|^{\sigma} U^q=0 \;\;\text{ in }\;\; (0,\infty)\times\real^N.
\end{equation}
Consider now an initial condition $u_0\in L_+^\infty(\real^N)$ such that $u_0\le U$ in $\real^N$ and let $u$ be the corresponding solution to~\eqref{eq1}, \eqref{init.cond} given by Definition~\ref{def.wp}. Introducing $R_0 := (\|u_0\|_\infty/A)^{(m-q)/(\sigma+2)}$, it readily follows from~\eqref{eq1}, \eqref{wp0}, \eqref{interm30} and the choice of $R_0$ that $u$ and $U$ are weak solutions to~\eqref{eq1} in $(0,\infty)\times B(0,R_0)$ such that
\begin{equation*}
	u_0(x) \le U(x), \qquad x\in B(0,R_0)
\end{equation*}
and
\begin{equation*}
	u(t,x) \le \|u_0\|_\infty = U(x), \qquad (t,x)\in (0,\infty)\times\partial B(0,R_0).
\end{equation*}
We are then in a position to apply the comparison principle and conclude that $u(t,x)\le U(x)$ for $(t,x)\in [0,\infty)\times \bar{B}(0,R_0)$. In particular, $0 \le u(t,0) \le U(0)=0$ for $t\ge 0$ and the proof is complete.
\end{proof}

\bigskip

\noindent \textbf{Acknowledgements} This work is partially supported by the Spanish project PID2020-115273GB-I00. Part of this work has been developed during visits of R. G. I. to Institut de Math\'ematiques de Toulouse and of Ph. L. to Universidad Rey Juan Carlos, and both authors want to thank for the hospitality and support of the mentioned institutions.

\bibliographystyle{plain}

\begin{thebibliography}{}

\end{thebibliography}


\begin{thebibliography}{1}

\bibitem{Abd98}
U. G. Abdullaev, \emph{Instantaneous shrinking of the support of a solution of a nonlinear degenerate parabolic equation}, Mat. Zametki, \textbf{63} (1998), no.~3, 323--331 (Russian). Translation in Math. Notes, \textbf{63} (1998), no.~3-4, 285--292.

\bibitem{Belaud01}
Y. Belaud, \emph{Time-vanishing properties of solutions of some degenerate parabolic equations with strong absorption}, Adv. Nonlinear Stud., \textbf{1} (2001), no.~2, 117--152.

\bibitem{BNP82}
M. Bertsch, T. Nanbu and L. A. Peletier, \emph{Decay of solutions of a degenerate nonlinear diffusion equation}, Nonlinear Anal., \textbf{6} (1982), no.~6, 539--554.


\bibitem{CV96}
M. Chaves and J. L. V\'azquez, \emph{Nonuniqueness in nonlinear heat propagation: a heat wave coming from infinity}, Differential Integral Equations, \textbf{9} (1996), no.~3, 447--464.

\bibitem{CV99}
M. Chaves and J. L. V\'azquez, \emph{Free boundary layer formation in nonlinear heat propagation}, Comm. Partial Differential Equations, \textbf{24} (1999), no.~11-12, 1945--1965.

\bibitem{CVW97}
M. Chaves, J. L. V\'azquez and M. Walias, \emph{Optimal existence and uniqueness in a nonlinear diffusion-absorption equation with critical exponents}, Proc. Roy. Soc. Edinburgh Sect. A, \textbf{127} (1997), no.~2, 217--242.


\bibitem{EK79}
L. C. Evans and B. F. Knerr, \emph{Instantaneous shrinking of the support of nonnegative solutions to certain nonlinear parabolic equations and variational inequalities}, Illinois Math. J., \textbf{23} (1979), no.~1, 153--166.

\bibitem{GSV99a}
V. A. Galaktionov, S. I. Shmarev and J. L. V\'azquez, \emph{Second order interface equations for nonlinear diffusion with very strong absorption}, Comm. Contemporary Math., \textbf{1} (1999), no.~1, 51--64.

\bibitem{GSV99b}
V. A. Galaktionov, S. I. Shmarev and J. L. V\'azquez, \emph{Regularity of interfaces in diffusion processes under the influence of strong absorption}, Arch. Rational Mech. Anal., \textbf{149} (1999), no.~3, 183--212.

\bibitem{GV94}
V. A. Galaktionov and J. L. V\'azquez, \emph{Extinction for a quasilinear heat equation with absorption I. Technique of intersection comparison}, Comm. Partial Diff. Equations, \textbf{19} (1994), no.~7-8, 1075--1106.

\bibitem{Gl1993} A.L. Gladkov, \emph{The Cauchy problem for certain degenerate quasilinear parabolic equations with absorption}, Sib. Math. J. \textbf{34} (1993), no.~1, 37--54.

\bibitem{Gl2001} A.L. Gladkov, \emph{The filtration-absorption equation with a variable coefficient}, Differ. Equ. \textbf{37} (2001), no.~1, 45--50.

\bibitem{GG2002}
A. L. Gladkov and M. Guedda, \emph{Diffusion-absorption equation without growth restrictions on the data at infinity}, J. Math. Anal. Appl., \textbf{274} (2002), no.~1, 16--37.

\bibitem{IL13}
R. G. Iagar and Ph.~Lauren\ced{c}ot, \emph{Existence and uniqueness of very singular solutions for a fast diffusion equation with gradient absorption},  J. London Math. Soc., \textbf{87} (2013), 509-529.

\bibitem{IMS22}
R. G. Iagar, A. I. Mu\~{n}oz and A. S\'anchez, \emph{Self-similar blow-up patterns for a reaction-diffusion equation with weighted reaction in general dimension}, Comm. Pure Appl. Anal., \textbf{21} (2022), no.~3, 891--925.

\bibitem{IMS22b}
R. G. Iagar, A. I. Mu\~{n}oz and A. S\'anchez, \emph{Self-similar solutions preventing finite time blow-up for reaction-diffusion equations with singular potential}, Submitted (2021), Preprint ArXiv no.~2111.04806.

\bibitem{IS20}
R. G. Iagar and A. S\'anchez, \emph{Self-similar blow-up profiles for a reaction-diffusion equation with strong weighted reaction}, Adv. Nonlinear Studies, \textbf{20} (2020), no.~4, 867--894.

\bibitem{IS22}
R. G. Iagar and A. S\'anchez, \emph{Self-similar blow-up profiles for a reaction-diffusion equation with critically strong weighted reaction}, J. Dynam. Differential Equations, to appear (2022), online DOI https://doi.org/10.1007/s10884-020-09920-w.

\bibitem{Ka74}
A. S. Kalasnikov, \emph{The propagation of disturbances in problems of non-linear heat conduction with absorption}, U.S.S.R. Comput. Math. Math. Phys., \textbf{14} (1975), no.~4, 70--85.

\bibitem{Ka84}
A. S. Kalashnikov, \emph{Dependence of properties of solutions of parabolic equations on unbounded domains on the behavior of coefficients at infinity}, Mat. Sb., \textbf{125 (167)} (1984), no.~3, 398--409 (Russian). Translated as Math. USSR Sb., \textbf{53} (1986), no.~2, 399--410.

\bibitem{KP86}
S. Kamin and L. A. Peletier, \emph{Large time behavior of solutions of the porous media equation with absorption}, Israel J. Math., \textbf{55} (1986), no.~2, 129--146.

\bibitem{KPV89}
S. Kamin, L. A. Peletier and J. L. V\'azquez, \emph{Classification of singular solutions of a nonlinear heat equation}, Duke Math. J., \textbf{58} (1989), no.~3, 601--615.

\bibitem{KU87}
S. Kamin and M. Ughi, \emph{On the behavior as $t\to\infty$ of the solutions of the Cauchy problem for certain nonlinear parabolic equations}, J. Math. Anal. Appl., \textbf{128} (1987), no.~2, 456--469.

\bibitem{KV88}
S. Kamin and L. V\'eron, \emph{Existence and uniqueness of the very singular solution of the porous media equation with absorption}, J. Analyse Math., \textbf{51} (1988), 245--258.

\bibitem{Kwak98}
M. Kwak, \emph{A porous media equation with absorption. I. Long time behavior}, J. Math. Anal. Appl., \textbf{223} (1998), no.~1, 96--110.

\bibitem{Le97}
G. Leoni, \emph{On very singular self-similar solutions for the porous media equation with absorption}, Differential Integral Equations, \textbf{10} (1997), no.~6, 1123--1140.

\bibitem{MPV91}
J. B. McLeod, L. A. Peletier and J. L. V\'azquez, \emph{Solutions of a nonlinear ODE appearing in the theory of diffusion with absorption}, Differential Integral Equations, \textbf{4} (1991), no.~1, 1--14.

\bibitem{Ot1996} F. Otto, \emph{$L^1$-contraction and uniqueness for quasilinear elliptic-parabolic equations}, J. Differential Equations \textbf{131} (1996), no.~1, 20--38.

\bibitem{PT86}
L. A. Peletier and D. Terman, \emph{A very singular solution of the porous media equation with absorption}, J. Differential Equations, \textbf{65} (1986), no.~3, 396--410.


\bibitem{Shi04}
P. Shi, \emph{Self-similar very singular solution to a $p$-Laplacian equation with gradient absorption: existence and uniqueness}, J. Southeast Univ., \textbf{20} (2004), no.~3, 381--386.

\bibitem{VW1994} J. L. V\' azquez and M. Walias, \emph{ Existence and uniqueness of solutions of diffusion-absorption equations with general data}, Differential Integral Equations \textbf{7} (1994), no.~1, 15--36.

\bibitem{YeYin}
H. Ye and J. Yin, \emph{Uniqueness of self-similar very singular solution for non-Newtonian polytropic filtration equations with gradient absorption}, Electronic J. Differential Equations, \textbf{2015} (2015), no.~83, 1--9.

\end{thebibliography}

\end{document}